\documentclass[a4paper,11pt,reqno]{amsart}

\usepackage{a4wide}
\usepackage{graphicx}
\usepackage{latexsym}
\usepackage{epsfig}
\usepackage{amssymb}
\usepackage{amstext}
\usepackage{amsgen}
\usepackage{amsxtra}
\usepackage{amsgen}
\usepackage{amsthm}
\usepackage{subfigure}
\usepackage{color}
\usepackage{enumitem}
\usepackage[foot]{amsaddr}
\usepackage{natbib}

\newtheorem{theorem}{Theorem}[section]
\newtheorem{prop}[theorem]{Proposition}
\newtheorem{lemma}[theorem]{Lemma}

\newtheorem{rem}[theorem]{Remark}

\numberwithin{equation}{section}

\newcommand{\N}{\mathbb{N}}

\newcommand{\R}{\mathbb{R}}

\newcommand{\Loneloc}{L^1_{loc}}
\newcommand{\Ltwoloc}{L^2_{loc}}
\newcommand{\Honeloc}{H^1_{loc}}

\newcommand{\feps}{f^\varepsilon}
\newcommand{\geps}{g^\varepsilon}
\newcommand{\fk}{f^k}
\newcommand{\gk}{g^k}
\newcommand{\muk}{\mu^k}
\newcommand{\aeps}{\alpha^\varepsilon}
\newcommand{\beps}{\beta^\varepsilon}
\newcommand{\veps}{v^\varepsilon}
\newcommand{\anull}{\alpha^0}
\newcommand{\bnull}{\beta^0}
\newcommand{\ainfty}{\alpha^\infty}
\newcommand{\binfty}{\beta^\infty}

\begin{document}

\title{On a Boltzmann type price formation model}

\author{Martin Burger$^1$}
\address{$^1$Institute for Computational and Applied Mathematics, University of M\"unster, Einsteinstrasse 62, 48149 M\"unster, Germany}

\author{ Luis Caffarelli$^2$}
\address{$^2$University of Texas at Austin, 1 University Station, C120, Austin, Texas 78712-1082, USA }

\author{Peter Markowich$^3$}
\address{$^3$ 4700 King Abdullah University of Science and Technology, Thuwal 23955-6900, Kingdom of Saudi Arabia  }

\author{Marie-Therese Wolfram$^4$}
\address{$^4$ Department of Mathematics, University of Vienna, Nordbergstr. 15, 1090 Vienna, Austria}
\maketitle

\begin{abstract}
In this paper we present a Boltzmann type price formation model, which is motivated by a parabolic free boundary model for the evolution of the prize presented by Lasry and Lions in 2007.
We  discuss the mathematical analysis of the Boltzmann type model and show that its solutions converge to solutions of the model by Lasry and Lions as the transaction rate
tends to infinity. Furthermore we analyse the behaviour of the initial layer on the fast time scale and illustrate the price dynamics with various numerical experiments.\\
{\bf Keywords: Boltzmann type equation, price formation, free boundary, asymptotics, numerical simulations\\ } 
\end{abstract}

\section{A Boltzmann type model for price formation}

\noindent Market microstructure analysis studies the trading mechanisms of assets in (financial) markets.
According to \cite{ohara} markets have two principal functions - they provide 
liquidity and facilitate the price. The evolution of the price is influenced by the trading 
system and the nature of the players, and is hence an emergent phenomenon from microscopic interaction, that calls for advanced mathematical modelling and analysis.

\noindent While there has been a lot of research in economics and statistics related to either the micro- or macrostructure of markets, few investigations dealt with systematically analysing the micro-macro transition. Recent work by \cite{lasrylions} revived the mathematical interest by introducing a price formation model that describes the evolution of price by a system of parabolic equations for the trader densities (as functions of the bid-ask price), with the agreed price entering as a free boundary. The authors motivated the model using mean field game theory, but the detailed microscopic origin remained unclear.

\noindent In this paper we provide a simple agent based trade model with standard stochastic price fluctuations together with discrete trading events. By modelling trading events between vendors and buyers as kinetic collisions we obtain a Boltzmann-type model for the densities. Then we prove rigorously that in the limit of large trading frequencies, the proposed Boltzmann model converges to the Lasry and Lions
free boundary problem. We also analyse other asymptotics beyond the scales that the free boundary model can describe. Hence we provide a basis for deriving macroscopic limits from the microscopic structure of trading events, which allows for various generalisations to make the model more realistic.

\noindent To set up the model we consider the price formation process of a certain good which is traded between two groups, namely a group of buyers and a group of vendors. The groups are described by the two positive densities $f = f(x,t)$ and $g = g(x,t)$,
where $x$ denotes the bid resp. ask price and $t$ the time. If a buyer and a seller agree on a price, a transaction takes place and the buyer immediately becomes a seller and vice versa. The transaction fee is denoted by a positive
 constant $a \in \R$. As a consequence the actual price for a buyer is $x+a$, therefore he/she will try to resell the good for at least that price. The profit for the vendor is $x-a$, therefore he/she is not willing to pay a higher
price in the next trading event.\\
The situation above can be described by the following Boltzmann type price formation model:
\begin{subequations}\label{e:boltzprice}
\begin{align}
f_t(x,t) &=  \frac{\sigma^2}{2} f_{xx}(x,t) - k f(x,t) g(x,t) + k f(x+a,t) g(x+a,t) \\
g_t(x,t) &=  \frac{\sigma^2}{2} g_{xx}(x,t) - k f(x,t) g(x,t) + k f(x-a,t) g(x-a,t).
\end{align}
with initial data
\begin{align}
f(x,0) = f_I(x) \geq 0, ~~g(x,0) = g_I(x) \geq 0,
\end{align}
\end{subequations}
independent of $k$. Here $f = f(x,t)$ denotes the density of buyers, $g = g(x,t)$ the density of vendors and $k$ the transaction rate. The parameter $a$ is the transaction cost for buyers and vendors, $\sigma$ the diffusivity. 
Note that the densities $f$ and $g$ depend on the parameter $k$. The volume of transactions at a price $x$ is given by: 
\begin{align}\label{e:trade}
\mu(x,t) = k f(x,t) g(x,t).
\end{align}
Note that we have conservation of the number of buyers and vendors, i.e.
\begin{align*}
\int_{\R} f(x,t) dx = \int_{\R} f_I(x) dx \text{ and } \int_{\R} g(x,t) dx = \int_{\R} g_I(x) dx,
\end{align*}
for all times $t > 0$.

\noindent The mathematical modelling of \eqref{e:boltzprice} was inspired by a price formation model presented by \cite{lasrylions}. They considered the same situation described above, but proposed a parabolic free 
boundary problem to model the evolution of the price. The model by Lasry $\&$ Lions reads:
\begin{subequations}\label{e:lasrylions}
\begin{align}
f_t(x,t) &= \frac{\sigma^2}{2} f_{xx}(x,t) + \lambda(t) \delta(x-p(t)+a) \text{ for } x < p(t) \text{ and } f(x,t) = 0 \text{ for } x > p(t)\\
g_t(x,t) &= \frac{\sigma^2}{2} g_{xx}(x,t) + \lambda(t) \delta(x-p(t)-a) \text{ for } x > p(t) \text{ and } g(x,t) = 0 \text{ for } x < p(t),
\end{align}
\end{subequations}
where the free boundary $p = p(t)$ denotes the agreed price of the trading good at time $t$ and $\lambda(t) = - f_x(p(t),t) = g_x(p(t),t)$. Note that the difference of buyer and vendor densities $v = f-g$ satisfies the following equation
\begin{align}\label{e:lasrylionsdiff}
v_t(x,t) &= v_{xx}(x,t) + \lambda(t)(\delta(x-p(t)+a)-\delta(x-p(t)-a)).
\end{align}
The function $\lambda(t)$ is the transaction rate at time $t$, which corresponds to the flux of buyers and vendors. The Dirac deltas correspond to trading events which take place at the agreed price $p = p(t)$, shifted by
the transaction cost $a$. The analysis of system \eqref{e:lasrylions} was studied in a number of papers, see \cite{MMPW, CGGK,CMP, CMW}. \\
The main difference between system \eqref{e:boltzprice} and \eqref{e:lasrylions} is that the agreed price $ p=p(t)$ enters as a free boundary in \eqref{e:lasrylions}. Furthermore the density of buyers and vendors are zero if the
price is greater or smaller than the agreed price. We shall show that solutions of the Boltzmann type equation \eqref{e:boltzprice} converge to solutions of \eqref{e:lasrylions} as the transaction rate tends to infinity, i.e. $k \rightarrow \infty$, thus giving a mathematical justification for \eqref{e:lasrylions} as a pure formation model. \\
This paper is organised as follows: we start with a detailed presentation of the mathematical modelling of \eqref{e:boltzprice} in Section \ref{s:modelboltz}. In Section \ref{s:kinfty} we show that solutions of \eqref{e:boltzprice} converge to solutions of \eqref{e:lasrylions} as $k\rightarrow \infty$. The initial layer problem for \eqref{e:boltzprice} is discussed in Section \ref{s:initlayer}, the scaling limit $k\rightarrow \infty, ~a\rightarrow 0$ and $ka = c > 0$ in Section \ref{s:kac}. 
This limit corresponds to high frequency trading, where computers trade goods on a rapid basis with little or no transaction costs involved. The behaviour of both models is illustrated by numerical simulations in Section \ref{s:numsim}.

\section{Modelling: From agent behaviour to Boltzmann-type equations} \label{s:modelboltz}

In order to describe price formation at a reasonably simple agent-based level, we consider a setup of a large number, say $N$, of buyers and a large number $M$ of vendors. With the exception of discrete events, the price changes are subject to random fluctuations, which we model as Brownian motions with diffusivity $\sigma$. The trading events can be modelled like kinetic collisions. After the usual kinetic limit $N \rightarrow \infty$, $M \rightarrow \infty$ we associate the density $f = f(x,t)$ to the group of buyers at time $t$, and the density $g = g(x,t)$ to the group of vendors at time $t$. Here $f(x,t)$ and $g(x,t)$ denote the fraction of the buyers and vendors willing to trade in an infinitesimal interval around $x$ at time $t$. If a buyer with state $x$ meets a vendor with state $y$, they will trade with a certain probability depending on their state. After collision, i.e. trade, a buyer becomes a vendor and vice versa. Due to conservation and indistinguishability we do not need to take into account the change of roles, but just model a standard collision $(x,y) \rightarrow (x',y')$ with collision kernel $K = K(x,y,x',y')$. More precisely, $K(x,y,x',y')$ counts the number of trading events per unit time of buyers willing to buy at price $x$ and reselling after the trading event at price $x'$, and the number of vendors willing to sell at price $y$ and re-buying after the trading event at price $y'$.  Using the notations for densities as above we arrive (with appropriate time scaling) at
\begin{subequations}\label{e:coll}
\begin{align}
 f_t(x,t) - \frac{\sigma^2}{2} f_{xx}(x,t) &= \int_\R \int_\R \int_\R K(x',y',x,y) f(x',t)g(y',t)~dy~dx'~dy'  \nonumber \\ 
&- \int_\R \int_\R \int_\R K(x,y,x',y') f(x,t)g(y,t) ~dy~dx'~dy' \\
 g_t(x,t) - \frac{\sigma^2}{2} g_{xx}(x,t) &= \int_\R \int_\R \int_\R K(x',y',y,x) f(y',t)g(x',t)~dx~dx'~dy'  \nonumber \\ 
& -\int_\R \int_\R \int_\R K(y,x,x',y') f(y,t)g(x,t) ~dx~dx'~dy'.
\end{align}
\end{subequations}
The peculiar aspect in the modelling of the collisions related to the Lasry-Lions approach is introduced by the transaction costs. Right after selling at a certain price, it only makes sense to re-buy at a price that is at least lower than the previous execution price minus the transaction costs, and similar reasoning applies to the buyers. Thus, if we assume symmetric transaction costs $a$ and denote by $r(x,y)$ the price at which a buyer with state $x$ and a vendor with state $y$ trade, then 
\begin{equation}
	x' = r(x,y) - a, \quad y' = r(x,y)+a, \label{collisionmodel}
\end{equation}
and hence $K$ is of the form
\begin{equation}
	K(x,y,x',y') = K_0(x,y) \delta(x'-r(x,y)+a) \delta(y'-r(x,y)-a), 
\end{equation}
where $\delta$ denotes the Dirac $\delta$-distribution. The detailed properties of $K_0$ and $r$ now depend on the modelling of the limit order book (cf. \cite{CL, CST,PGPS,BM, AJ}). Here we make a simple assumption that $x$ and $y$ can be interpreted as the bids of buyers and vendors, which are tried to be matched exactly. Hence, $K_0$ is centred around $x=y$, and the simplest choice is to assume 
\begin{equation}
	K_0(x,y) = k \delta(x-y), 
\end{equation}
where the constant $k$ is the trading frequency. We only need to specify 
\begin{align*}
r(x,x) = x.
\end{align*}
\noindent This leads to the collision kernel
\begin{align*}
K(x,y,x',y') = \delta(x-y) \delta(x'-r(x,y)+a)\delta(y'-r(x,y)-a).
\end{align*}
For smooth test functions $\varphi = \varphi(x,y,x',y')$ we have
\begin{align*}
\langle K, \varphi \rangle &= \int_\R \varphi(x,x,r(x,x)-a, r(x,x)+a) dx \\
&= \int_\R \varphi(x,x,x-a,x+a) dx. 
\end{align*}
Applying this to \eqref{e:coll} leads to \eqref{e:boltzprice}.
\noindent For large numbers $N$ and $M$ it appears natural to consider the case of large $k$, since a high potential for transactions is available. One observes that the volume of transactions in \eqref{e:boltzprice} is given by $k f g$, hence the scaled density $\rho = c f g$ with $c=\frac{1}{\int_\R fg~dx}$ yields a density of currently traded prices. If $\rho$ is well centred, the mean, median, or maximum $\rho$ will give an estimate of the price $p(t)$, which we will also exploit in numerical comparison with the Lasry-Lions model. For large $k$ we will see that $\rho$ concentrates to a Dirac $\delta$ distribution centred at $p(t)$.
Note that the mean supply and demand price satisfy
\begin{align*}
\frac{d}{dt} \int_\R x f(x,t) dx &= -a\int_\R k f(x,t) g(x,t) dx \leq 0 \\
\frac{d}{dt} \int_\R x g(x,t) dx &= a \int_\R k f(x,t) g(x,t) dx \geq 0.
\end{align*}
This means that the average demand price is decreasing, while the average supply price increases as long as $fg$ does not vanish. This shift is clearly inherent in the collision model for the trades, since each collision decreases the demand price of an agent and increases the supply price of another one.

\noindent We finally mention that the above modelling is reasonably simple to study basic effects, but the agent-based interpretation easily allows various generalisations that can make it more realistic. An obvious first example is replacing the Brownian motion by more general stochastic differential equations, which simply changes the differential operator to some other parabolic equation, but leaves the collisions unchanged. Even herding effects could be modeled this way, e.g. by making the drift of a buyer dependent on the (empirical) density $f$, e.g. on $\int_\R x f(x,t)~dx. $ The extension to more complicated trading models is inherent in our model by specifying $K_0$ and $r$. A simple alternative model is
\begin{equation}
	K_0(x,y) = k {\bf 1}_{\{y\leq x\}}, \qquad r(x,y) = \frac{x+y}2,  
\end{equation}
which corresponds to a self-organised trading with using the mean value of the prices when $y \leq x$. 

\section{Limiting equations as $k \rightarrow \infty$}\label{s:kinfty} 

\noindent In this section we show that the difference $f-g$ of solutions of the system \eqref{e:boltzprice} converges to the solution of \eqref{e:lasrylionsdiff} as the transaction rate $k$ tends to infinity and that, as $k\rightarrow \infty$, $f = (f-g)^+$ and $g = (f-g)^-$. To emphasise the dependence of $f$ and $g$ on $k$, we use the notation $\fk = \fk(x,t)$ and $\gk = \gk(x,t)$ throughout this section. Without loss of generality we set $\frac{\sigma^2}{2} = 1$ in the rest of the paper. Note that existence and uniqueness of a non-negative solutions of \eqref{e:boltzprice} for $k \geq 0$ follows trivially from the estimates stated in this section. For the following we assume that $f_I$ and $g_I$ are independent of $k$ and satisfy:
\begin{enumerate}[label=(\textit{\Alph*})]
\item Let $f_I, g_I \geq 0$ on $\R$ and $f_I, g_I \in \mathcal{S}(\R)$. \label{a:A}
\end{enumerate}
Note that assumption \ref{a:A} can be weakened, i.e. 
\begin{align*}
f_I, g_I \in L^1(\R) \cap L^\infty(\R)\cap H^1(\R) \text{ with sufficiently fast decay as } \lvert x \rvert  \rightarrow \infty,
\end{align*} 
at the expense of additional technicalities.
\begin{prop}
Let assumption \ref{a:A} be satisfied. Then system \eqref{e:boltzprice} has unique positive solutions $\fk, \gk \in L^{\infty}(0,\infty; L^1(\R) \cap L^\infty(\R))$ bounded uniformly as $k \rightarrow \infty$. Furthermore 
\begin{align}\label{e:conv}
\fk \gk \rightharpoonup 0 \text{ in } \mathcal{D}'(\mathbb{R} \times [0,\infty]).
\end{align}
\end{prop}
\begin{proof}
\noindent Note that the function $u(x,t) = \fk(x,t) + \gk(x+a,t)$ satisfies the IVP for the heat equation 
\begin{align}\label{e:u}
u_t(x,t) &= u_{xx}(x,t) \\
u(x,0) &= f_I(x)+g_I(x+a).\nonumber
\end{align}
The heat equation has a unique positive solution $u = u(x,t) \in L^{\infty}(0,\infty; L^\infty(\R))$ for initial data $f_I, g_I \in \mathcal{S}(\R)$, $g_I(x) \geq 0, ~f_I(x) \geq 0$ independent of $k$. Hence we deduce that the weak limits $f^\infty$, $g^\infty$ satisfy
\begin{align*}
  f^\infty(\cdot ,t) \in L^1_+(\R) \cap L^\infty_x(\R) \text{ and } g^\infty(\cdot,t) \in L^1_+(\R) \cap  L^\infty_x(\R) \text{ uniformly in } t > 0.
\end{align*} 
To show the weak convergence of $\fk\gk$ to zero we rewrite the first equation in \eqref{e:boltzprice} as
\begin{align*}
\frac{1}{k}\fk_t(x,t) - \frac{1}{k}\fk_{xx}(x,t) &= - \fk(x,t)\gk(x,t) + \fk(x+a,t) \gk(x+a,t).
\end{align*}
We define $\zeta^k(x,t) = \fk(x,t) \gk(x,t)$, $\zeta^k \in L^\infty(0,\infty; L^1_+(\R) \cap L^\infty(\R))$. Then in the limit $k\rightarrow \infty$ we deduce that
\begin{align*}
-\zeta^k(x,t) + \zeta^k(x+a,t) = o(1) \in \mathcal{D}'(\R \times [0,\infty))  \quad \Rightarrow \quad \zeta^{\infty}(x,t) = \zeta^\infty(x+a,t). 
\end{align*}
The only periodic function in $L^1(\R)$ is the zero-function, which concludes the proof.
\end{proof}
\noindent Next we introduce as in \cite{CMP,CMW} the functions
\begin{align}\label{e:defFG}
F^k(x,t) = \sum_{l=0}^\infty \fk(x+al,t) \text{ and } G^k(x,t) = \sum_{l=0}^\infty \gk(x-al,t).
\end{align}
\begin{prop}\label{p:FG}
Let assumption \ref{a:A} be satisfied. The functions $F^k$ and $G^k$ defined by the infinite series \eqref{e:defFG} satisfy $F^k,G^k \in L^{\infty}(0,\infty; L^\infty(\R))$. The series converge locally uniform in $x$ and $t$ as $k \rightarrow \infty$.
\end{prop}
\begin{proof}
\noindent Since $\fk(\cdot,t),~\gk(\cdot,t) \in \mathcal{S}(\R)$ we can estimate the series $F^k$  by the following integral
\begin{align*}
I(x) = \int_0^\infty \frac{c}{(1+\lvert x + y \rvert) ^ 2} dy = \int_x^\infty \frac{c}{(1 + \lvert z \rvert )^2} dz \text{ for some } c>0.
\end{align*}
For $x > 0$ the integral $I(x) = \frac{1}{q-1}\frac{1}{(1+x)}$ tends to $0$ as $x \rightarrow \infty$. In the case $x < 0$ we obtain the following estimate for the integral\begin{align*}
I(x) &= \int_{-\lvert x \rvert}^0 \frac{1}{(1+z)^2} dz + \int_0^\infty \frac{1}{(1+z)^2} dz\\
&\leq c_1 + \frac{c_2}{(1+\lvert x \rvert)}, \text{ where } c_1, c_2 \text{ are independent of x.} \qedhere
\end{align*}
\end{proof}
\noindent Note that $F^k$ and $G^k$ satisfy: 
\begin{subequations}\label{e:FG}
\begin{align}
F^k_t(x,t) &= -k \fk(x,t) \gk(x,t) + F^k_{xx}(x,t)\\
G^k_t(x,t) &= -k \fk(x,t) \gk(x,t) + G^k_{xx}(x,t).
\end{align}
\end{subequations}
\noindent Then the difference $\Phi :=F^k-G^k$ satisfies the heat equation $\Phi_t = \Phi_{xx}$. Note that $\Phi$ is independent of $k$ for $f_I, g_I$ independent of $k$.
From system \eqref{e:FG} we deduce that $\muk = k \fk \gk$ converges as $k \rightarrow \infty$ to a locally bounded positive measure on $\R \times [0,\infty)$.  Then the limiting functions $f^\infty$ and $g^\infty$ satisfy the following system
\begin{subequations}\label{e:limitfg}
\begin{align}
f^\infty_t(x,t) &= -\mu^\infty(x,t) + \mu^\infty(x+a,t) +f_{xx}^\infty(x,t) \\
g^\infty_t(x,t) &= -\mu^\infty(x,t) + \mu^\infty(x-a,t) + g_{xx}^\infty(x,t).
\end{align}
\end{subequations}
Note that the difference $v= f^{\infty}-g^{\infty}$ satisfies a parabolic PDE with a similar structure as the price formation model \eqref{e:lasrylionsdiff}, i.e.
\begin{align}\label{e:difflimit}
v_t(x,t) = v_{xx}(x,t) + \mu^\infty(x+a,t) - \mu^\infty(x-a,t).
\end{align}

\subsection{A priori estimates} We derive various a priori estimates for the solutions $\fk$ and $\gk$ of system \eqref{e:boltzprice}, which will we used for the identification
of the limiting system as $k\rightarrow \infty$.
\begin{prop}
The solution of \eqref{e:boltzprice} satisfies the following a-priori estimate:
\begin{align}
\int_0^T \int_{\mathbb{R}} (\fk_x(x,t) + \gk_x(x,t))^2 dx dt \leq const, 
\end{align}
uniformly as $k \rightarrow \infty$.
\end{prop}

\begin{proof}
In the following calculations we neglect the dependence of the functions $\fk,~\gk, u, F^k$ on the variables $x$ and $t$, and only state their arguments if necessary.
We multiply system \eqref{e:boltzprice} with $\fk$ and $\gk$ respectively and obtain
\begin{align*}
\frac{1}{2} \frac{d}{dt} \int (\fk)^2 dx &= - k \int (\fk)^2 \gk~ dx + k \int \fk(x+a,t) \fk(x,t) \gk(x+a,t)~dx - \int (\fk_x)^2~dx,\\
\frac{1}{2} \frac{d}{dt} \int ( \gk)^2 dx &= - k \int \fk (\gk)^2~dx + k \int \fk(x-a,t) \gk(x,t) \gk(x-a,t)~dx - \int (\gk_x)^2~dx.
\end{align*}
With changes of variables in the second integral on the right hand side of both equations we deduce for the sum of both equations that
%we deduce that
%\begin{align*}
%\frac{1}{2} \frac{d}{dt} \int f^2(x) dx &= - k \int f^2(x) g(x) dx + k \int f(x) g(x) f(x-a) dx - \int f_x^2(x) dx,
%\end{align*}
%and 
%\begin{align*}
%\frac{1}{2} \frac{d}{dt} \int g^2(x) dx &= - k \int f(x) g^2(x) dx + k \int f(x) g(x) g(x+a) dx - \int g_x^2(x) dx.
%\end{align*}
%Then the sum of both equations satisfies
\begin{align*}
\frac{1}{2} \frac{d}{dt} \int ((\fk)^2+(\gk)^2) dx &= -k \int [ \fk\gk(\fk+\gk) - \fk\gk(\fk(x-a,t)+\gk(x+a,t))] dx \\
&\phantom{-k} - \int (\gk_x+\fk_x)^2 dx.
\end{align*}
Next we use that $\gk(x+a,t) = u(x,t) - \fk(x,t) $ and $\fk(x-a,t) = u(x-a,t) - \gk(x,t)$ and obtain
\begin{align}\label{e:diff}
\begin{split}
  \frac{1}{2} \frac{d}{dt} \int& ((\fk)^2+(\gk)^2) dx = -k \int [ \fk\gk(\fk+\gk) + (\fk)^2\gk + \fk(\gk)^2] dx  \\
&\phantom{=}+\int k \fk\gk (u(x,t) - u(x-a,t)) dx - \int (\gk_x+\fk_x)^2 dx.
\end{split}
\end{align}
Multiplication of \eqref{e:FG} by $u = u(x,t)$ and, resp., by $u(x+a,t)$ gives
\begin{align*}
\int (F^k u)_t dx - \int F^k u_t dx = \int F^k u_{xx} dx - k \int \fk \gk u dx.
\end{align*}
Note that all integrals are well defined, since $u = u(x,t)$ decays algebraically as $\lvert x \rvert \rightarrow \infty$.
Hence we obtain after integration over the time interval $(0,T)$ that
\begin{align}\label{e:est1}
\int_0^T \int_{\R} k \fk(x,t) \gk(x,t) (u(x,t) + u(x+a,t))~dx~dt  \leq const 
\end{align}
uniformly as $k \rightarrow \infty$. Using \eqref{e:est1} in \eqref{e:diff} concludes the proof.
\end{proof}

\noindent Multiplying the first equation in system \eqref{e:FG} by a test function $\varphi \in C_0^\infty(\mathbb{R} \times [0,\infty)), \varphi \geq 0$, we obtain
\begin{align*}
  \iint k \fk \gk \varphi dx dt = \iint F^k \varphi_t dx dt + \int F^k \varphi \mid_{t=0}^T dx + \iint F^k \varphi_{xx} dx dt \leq const.
\end{align*} 
If $\varphi(\cdot,t)$ has compact support on $(-3R,3R)$ and $\varphi \equiv 1$ in $(-2R, 2R)$ we deduce that
\begin{align*}
\int_0^T \int_{\lvert x \rvert < R}k \fk(x,t) \gk(x,t) dx dt \leq const
\end{align*}
uniformly in $k$. Then the functions $\fk$ and $\gk$ satisfy
\begin{align}\label{e:regtime}
\fk_t, \gk_t \in \Loneloc(\R \times [0,\infty)) + \Ltwoloc(0,\infty; H^{-1}(\R)), 
\end{align}
and 
\begin{align}
\fk,\gk \in \Ltwoloc(0,\infty; \Honeloc(\R))
\end{align}
uniformly in $k$.

\subsection*{Strong convergence} Let $\Omega$ be an interval on the real line $\Omega = (-R,R)$. To show strong convergence of \eqref{e:conv} we use a generalised version of the Aubin-Lions lemma, cf. \cite{Lions}. We define the following 
 spaces
\begin{align}
  \underbrace{H^1 (\Omega)}_{= B_0} \subset \subset \underbrace{L^2(\Omega)}_{=B} \subseteq \underbrace{H^{-1}(\Omega)}_{= B_1}.
\end{align}
For a function $v \in L^1(0,T; H^{-1}(\Omega))$ we estimate the corresponding norm by
\begin{align*}
  \int_0^T \sup_{\varphi \in H^1_0(\Omega)} \frac{\int u \varphi dx}{\lVert \varphi \rVert_{H^1_0(\Omega)}} dx~dt &\leq \int_0^T \sup_{\varphi \in H^1_0(\Omega)} \frac{\int \lvert u \rvert  dx~dt \lVert \varphi \rVert_{L^\infty(\Omega)} }{\lVert \varphi \rVert_{H^1_0(\Omega)}} dx~dt \\
& \leq c \int_0^T \int_{\Omega} \lvert u (x,t) \rvert dxdt.
\end{align*}
 Due to \eqref{e:regtime} and the previous estimate we deduce that $\fk_t, \gk_t \in L^1(0,T; H^{-1}(\Omega))$. Since $\fk$ and $\gk$ converge weakly to $f^\infty$ and $g^\infty$ in $L^2(0,T; L^2(\Omega))$, we use the general version
of the Aubin-Lions lemma, cf. \cite{Lions}, to obtain that
\begin{align}
\lim_{k \rightarrow \infty} (\fk, \gk) \rightarrow (f^\infty, g^\infty) \text{ in } (L^2(0,T; L^2(\Omega)))^2,
\end{align} 
for every bounded interval $\Omega$. Hence we conclude strong convergence of the product in $\Loneloc (0,\infty;\R)$, i.e.
\begin{align}\label{e:stronglim}
\lim_{k \rightarrow \infty} \fk \gk = f^\infty g^\infty = 0.
\end{align}

\subsection{The limiting equations as $k\rightarrow \infty$} Now we are able to identify the price formation model given by Lasry $\&$ Lions \eqref{e:lasrylions} in the
limiting case $k\rightarrow \infty$. We have already shown that the Boltzmann type price formation model has a similar structure in the limit $k \rightarrow \infty$, see \eqref{e:difflimit}.
Therefore it remains to show that the measures on the right hand side of \eqref{e:difflimit} are indeed the transaction rates defined by Lasry $\&$ Lions. We shall do this for the
 initial data $f_I$ and $g_I$, which satisfies the compatibility conditions necessary for \eqref{e:lasrylions}, i.e.
\begin{enumerate}[label=(\textit{\Alph*}), start=2]
\item  $\sup \lbrace f_I(x) > 0 \rbrace \leq \inf \lbrace g_I(x) > 0 \rbrace$. \label{a:wellprep}
\end{enumerate}  
We  find by applying the Lebesgue's dominated convergence theorem that
\begin{align*} 
F^\infty(x,t) = \sum_{l=0}^\infty f^\infty(x + al,t) \text{ and  } G^\infty(x,t) = \sum_{l=0}^\infty g^\infty(x-al,t).
\end{align*}
Then the difference $\Phi(x,t) = F^\infty(x,t)-G^\infty(x,t)$ solves the IVP for the heat equation 
\begin{subequations}
\begin{align}
\Phi_t(x,t) &= \Delta \Phi(x,t)\\
\Phi(x,0) &= \Phi_I(x) := \sum_{l=0}^\infty f_I(x+al) - \sum_{l=0}^\infty g_I(x-al).
\end{align}
\end{subequations}
\noindent We start with an additional regularity result.
\begin{prop}
The solutions $\fk$ and $\gk$ of \eqref{e:boltzprice} satisfy $\fk,\gk \in L^1(0,\infty; C_b(\R))$ for almost every time $t>0$. This implies the continuity of 
$F^k$ and $G^k$ for almost every time $t>0$. 
\end{prop}
\begin{proof}
The proposition follows from the compact embedding of $H^1(\R)$ into $C_b(\R)$, therefore
\begin{align*}
&  \fk,\gk \in L^1(0,\infty; C_b(\R)) \text{ for almost every time } t. \hspace*{2cm} \qedhere
\end{align*}
\end{proof}
\noindent We now prove the main result of this section.
\begin{theorem} \label{t:id}
Let $f_I$ and $g_I$ satisfy assumption \ref{a:A}. 
\begin{enumerate}
\item[(i)] Then the limiting functions $f^\infty$ and $g^\infty$ are given by:
\begin{align*}
&f^\infty(x,t) = \Phi^+(x,t)-\Phi^+(x+a,t), \text{ and } g^\infty(x,t) = \Phi^-(x,t) - \Phi^-(x-a,t).
\end{align*}
\item[(ii)] Additionally let \ref{a:wellprep} hold. Then $f^\infty$ and $g^\infty$ satisfy system \eqref{e:lasrylions}.
\end{enumerate}
\end{theorem}
\noindent The main ingredient of the proof is:
\begin{lemma}\label{p:identifylimit}
Let $f_I$ and $g_I$ satisfy \ref{a:A}. Then
\begin{enumerate}
\item [(i)] $F^\infty = (F^\infty-G^\infty)^+$ and $G^\infty = (F^\infty-G^\infty)^-$.
\item [(ii)] Additionally let \ref{a:wellprep} hold. Then there exists a unique globally defined continuous function $p = p(t)$ (the price), which satisfies
\begin{align*}
  F^\infty(p(t),t) = G^\infty(p(t),t) = 0 \text{ and } 
\end{align*}
\end{enumerate}
\end{lemma}
\begin{proof}
(i) Take $(x,t) \in \R \times (0,\infty)$  such that $F^\infty(\cdot,t), G^\infty(\cdot,t)$ are continuous and $F^\infty(x,t) > 0$. Then there exists an $l$ such that $f^\infty(x + al,t) > 0$ and therefore $g^\infty(x + al, t) = 0$, due to \eqref{e:stronglim}.
The function $u(x,t) = \fk(x,t) + \gk(x+a,t)$ satisfies the heat equation \eqref{e:u} and is positive, hence we can deduce that $f^\infty(x-a(l-1), t) > 0$ and subsequently that $g^\infty(x + a(l-1),t)=0$.
By repeating this argument we can show that $g^\infty(x-al,t) = 0$ for $l=0,1,2, \ldots$, which implies $G^\infty(x,t) = 0$. Analogously we prove $G^\infty(x,t) > 0$ implies $F^\infty(x,t)$ = 0.\\
(ii) From the results in the work by \cite{CMP, CMW} we deduce that there exists a unique globally defined continuous function $p = p(t)$ (the price) which satisfies
\begin{align}
\Phi(p(t),t) = 0 = F^{\infty}(p(t),t) - G^\infty(p(t) ,t).
\end{align}
Therefore $F^\infty(p(t),t) = G^\infty(p(t),t)$. Assuming that $F^\infty(p(t),t) > 0$ leads to a contradiction by proceeding similarly to the proof of (i). 
\end{proof}

\noindent This allows us to identify the limiting functions $f^\infty$ and $g^\infty$ in the case of (ii) of Theorem \ref{t:id}  and conclude, as in \cite{CMP, CMW} that 
$f^\infty$ and $g^\infty$ satisfy system \eqref{e:lasrylions}. \\
 We remark that the theory developed in \cite{CMP, CMW} applies strictly speaking only if $\sup \lbrace f_I(x) > 0 \rbrace = \inf \lbrace g_I(x) > 0 \rbrace$, however an extension
for $\sup \lbrace f_I(x) > 0 \rbrace \leq \inf \lbrace g_I(x) > 0 \rbrace$ is straight forward. 

\begin{rem}
\noindent In the case (ii) the transaction volume in the limit $k \rightarrow \infty$ is 
\begin{align}\label{e:muinf}
\mu^\infty (x,t) = \lambda(t) \delta(x-p(t))
\end{align}
and the limited density of traded prices is $\rho^\infty (x,t) = \delta (x-p(t))$. To see this we compare \eqref{e:difflimit} to \eqref{e:lasrylionsdiff} and obtain:
\begin{align*}
\mu^\infty (x+a,t) - \mu^\infty(x-a,t) = \lambda (t) (\delta(x-p(t) +a) - \delta(x-p(t) -a).
\end{align*}
This implies that $\mu^\infty (x,t) = \lambda(t) \delta (x-p(t)) + A(t, x)$,
where $A$ is a nonnegative and $2a$-periodic function. Now we expand $A$ into its Fourier series and consider a single
harmonic term in the series, given by $a_l(t) e^{ i \pi x \frac{l}{a}}, l \in \mathbb{Z}$. When taking this term as inhomogeneity in the heat equation
\begin{align*}
&z_t = z_{xx} - a_l(t) e^{ i \pi x \frac{l}{a}}\\
&z(x,t=0) = 0,
\end{align*}
we easily compute the solution:
\begin{align*}
z(x,t) &= - e^{i \pi x \frac{l}{a}} \int_0^t e^{-(\frac{\pi l}{a})^2 (t-s)} a_l(s) ds.
\end{align*}
Passing to the limit k to infinity in \eqref{e:FG}(a) gives the heat equation for $F^\infty$ with $-\mu^\infty$ as inhomogeneity. From the proof of Proposition \ref{p:FG} 
it becomes clear that $F^\infty$ tends to $0$ for $x \rightarrow \infty$ (for every fixed t) and thus $F^\infty$ does not admit $x$-periodic modes. We conclude \eqref{e:muinf}.
\end{rem}

If assumption \ref{a:wellprep} does not hold, then it can be shown that the local transaction rate satisfies in the limit $k \rightarrow \infty$:
\begin{align*}
\mu^\infty(x,t) = \sum_{j \in J(t)} \lvert \Phi_x(p_j(t),t) \rvert \delta (x-p_j(t)),
\end{align*}
where $J(t)$ denotes the index set (finite or countably finite) such that $\lbrace p_j(\cdot,t) \mid j \in J(t) \rbrace$  is the set of zeros of $\Phi(\cdot,t)$.

\section{The initial layer problem - preparation of the initial data}\label{s:initlayer}

In this section we discuss the behaviour of the initial layer on the fast time scale. This initial layer occurs if $(f^\infty(x, t=0), g^\infty (x, t=0))$ as computed in Theorem \ref{t:id} (i) 
differs from $(f_I, g_I)$. \\
Let $\varepsilon = \frac{1}{k}$, then system \eqref{e:boltzprice} reads
\begin{align*}
  \varepsilon \feps_t(x,t) &= - \feps(x,t) \geps(x,t) + \feps(x+a,t) \geps(x+a,t) + \varepsilon \feps_{xx}(x,t)\\
\varepsilon \geps_t(x,t) &= -\feps(x,t) \geps(x,t) + \feps(x-a,t) \geps(x-a,t) + \varepsilon \geps_{xx}(x,t).
\end{align*}
Let $\tau = \frac{t}{\varepsilon}$ denote the fast time scale and the corresponding fast-scale dependent variables by 
\begin{align*}
\aeps(x,\tau) := \feps(x,t) \text{ and } \beps(x,\tau) := \geps(x,t).
\end{align*}
Then $\aeps = \aeps(x,\tau)$ and $\beps=\beps(x,\tau)$ satisfy 
\begin{subequations}\label{e:abeps}
\begin{align}
\aeps_\tau(x,\tau) &= - \aeps (x,\tau) \beps(x,\tau) + \beps(x+a,\tau) \aeps(x+a,\tau) + \varepsilon \aeps_{xx}(x,\tau)\label{e:aeps}\\
\beps_\tau(x,\tau) &= -\beps(x,\tau) \aeps(x,\tau) + \beps(x-a,\tau) \aeps(x-a,\tau) + \varepsilon \beps_{xx}(x,\tau)\label{e:beps}.
\end{align}
with initial data 
\begin{align}\label{e:abepsinit}
\aeps(x,0) &= f_I(x) \text{ and } \beps(x,0) = g_I(x).
\end{align}
\end{subequations}

\begin{prop}
In the limit $\varepsilon \rightarrow 0$ the fast scale variables $\aeps = \aeps(x,\tau)$ and $\beps = \beps(x,\tau)$ converge to
\begin{align*}
\aeps \rightarrow \alpha^0, ~ \beps \rightarrow \beta^0 \text { in } \Ltwoloc(0,\infty; \Ltwoloc(\R)).
\end{align*} 
The limits $\alpha^0 = \alpha^0(x,\tau)$ and $\beta^0 = \beta^0(x,\tau)$ satisfy the ODE system
\begin{subequations}\label{e:priceode}
\begin{align}
\alpha^0_\tau(x,\tau) &= - \alpha^0(x,\tau) \beta^0(x,\tau) + \alpha^0(x+a,\tau) \beta^0(x+a,\tau)\\
\beta^0_\tau(x,\tau) &= -\alpha^0(x,\tau) \beta^0(x,\tau) + \alpha^0(x-a,\tau) \beta^0(x-a,\tau).
\end{align}
\end{subequations}
Furthermore the system \eqref{e:priceode} with IC \eqref{e:abepsinit} has a unique continuous space-time solution.
\end{prop}
\begin{proof}
Note that $\veps(x,\tau) = \aeps(x,\tau) + \beps(x+a,\tau)$ satisfies the heat equation
\begin{subequations} \label{e:v}
\begin{align}
\veps_\tau(x,\tau) & = \varepsilon \veps_{xx}(x,\tau)\label{e:veps} \\
\veps(x,0) &= f_I(x) + g_I(x+a).
\end{align}
\end{subequations}
Solutions of \eqref{e:veps} with initial datum $f_I(x),~g_I(x) \in \mathcal{S}(\R)$ decay algebraically fast and are smooth, and we deduce that $\aeps, ~\beps \in L^\infty(0,\infty;  L^1(\R) \cap L^\infty(\R))$
 uniformly as $\varepsilon \rightarrow 0$. Next we differentiate equation \eqref{e:aeps} with respect to $x$, multiply with $\aeps_x$ and integrate over $\R$ to obtain
\begin{align*}
\frac{1}{2}\frac{d}{d\tau} \int_\R (\aeps_x)^2 dx \leq K \int_\R (\aeps_x)^2 dx + K \int_\R (\beps_x)^2 dx - \varepsilon \int_\R (\aeps_{xx})^2 dx.
\end{align*}
The same holds for \eqref{e:beps}, i.e. 
\begin{align*}
\frac{1}{2}\frac{d}{d\tau} \int_\R (\beps_x)^2 dx \leq K \int_\R (\aeps_x)^2 dx + K \int_\R (\beps_x)^2 dx - \varepsilon \int_\R (\beps_{xx})^2 dx.
\end{align*}
Integration over $(0,T)$ gives, for arbitrary $T > 0$:
\begin{align*}
\int_{\R} [(\aeps_x)^2(x,\tau) + (\beps_x)^2(x,\tau)] dx \leq K_1(T),
\end{align*}
for $\tau \in (0,T]$. Similar calculations, i.e. differentiating equations \eqref{e:aeps} and \eqref{e:beps} with respect to $\tau$ and multiplication by $\aeps_\tau$ and $\beps_\tau$ respectively, lead to
\begin{align*}
\int_\R [(\aeps_\tau)^2(x,\tau) + (\beps_\tau)^2(x,\tau)] \leq K_2(T),
\end{align*}
for $\tau \in (0,T)$. Therefore $\aeps$ and $\beps$ are bounded in $\Honeloc(\R \times (0,\infty))$ uniformly as  $\varepsilon \rightarrow 0$. Due to the compact embedding of $\Honeloc$ into $\Ltwoloc$ we can deduce that,
after extraction of a subsequence 
\begin{align*}
\aeps \rightarrow \alpha^0,~\beps \rightarrow \beta^0 \text{ in }\Ltwoloc(0,\infty; \Ltwoloc(\R)).
\end{align*}
The existence of a unique continuous space-time solution of \eqref{e:priceode}, given by $(\anull, \bnull)$, is immediate.
\end{proof}

\noindent Next we study the behaviour of $\alpha^0 = \alpha^0(x,\tau)$ and $\beta^0 = \beta^0(x,\tau)$ as $\tau \rightarrow \infty$, which corresponds to the ``end of the initial layer''.
\begin{prop}
\begin{enumerate}[label=(\textit{\alph*})]
\item Solutions $\anull = \anull(x,\tau)$ and $\bnull = \bnull(x,\tau)$ of \eqref{e:priceode} converge in the limit $\tau \rightarrow \infty$
\begin{align*}
  \lim_{\tau \rightarrow \infty} \anull(x,\tau) = \ainfty(x),~~\lim_{\tau \rightarrow \infty} \bnull(x,\tau) = \binfty(x) \text{ and } \ainfty(x) \binfty(x) \equiv 0 \text{ on } \R.
\end{align*}
\item If $f_I(x) g_I(x) \equiv 0$ on $\R$ then $\anull(x,t) \equiv f_I(x)$ and $\bnull(x,t) \equiv g_I(x)$ on $\R \times [0,\infty)$.
\end{enumerate}
\end{prop}
\begin{proof}
Set $w^0 = \anull(x,\tau) \bnull(x,\tau)$ and compute
\begin{align*}
&w_\tau^0(x,\tau) = \anull(x,\tau) w^0(x-a,\tau) - [\anull(x,\tau) + \bnull(x,\tau)] w^0(x,\tau) + \bnull(x,\tau) w^0(x+a,\tau)\\
&w^0(x,0) = \anull(x,0) \bnull(x,0) = f_I(x) g_I(x).
\end{align*}
If $f_I(x) g_I(x) \equiv 0$ on $\R$, then $w^0(x,\tau) = \anull(x,\tau) \bnull(x,\tau) \equiv 0$. Hence \eqref{e:priceode} gives $\anull(x,\tau) = f_I(x)$ and $\bnull(x,\tau) = g_I(x)$ on $\R \times [0,\infty)$.\\
We observe that
\begin{align*}
\anull(x,\tau) + \bnull(x+a,\tau) \equiv f_I(x) + g_I(x+a)
\end{align*}
and 
\begin{align*}
  \underbrace{\sum_{l=0}^\infty \anull(x+al,\tau)}_{=: A^0(x,\tau)} - \underbrace{\sum_{l=0}^\infty \bnull(x-al,\tau)}_{=: B^0(x,\tau)} = \sum_{l=0}^\infty f_I(x+al,\tau) - \sum_{l=0}^\infty g_I(x-al,\tau).
\end{align*}
Note that $\anull(\cdot,\tau), \bnull(\cdot,\tau) \in \mathcal{S}^+(\R)$, thus both sequences $A^0(\cdot,\tau)$ and $B^0(\cdot,\tau)$ converge uniformly locally in $x$. Moreover both sequences are decreasing
in time, since
\begin{align}\label{e:A0}
A^0_\tau = -w^0 \leq 0 \textrm{ and } B^0_\tau = -w^0 \leq 0.
\end{align}
Therefore for every $x \in \R$ fixed, we have 
\begin{align*}
 A^0(x,\tau) \downarrow A^\infty(x) \text{ and } B^0(x,\tau) \downarrow B^\infty(x) \text{ as } \tau \rightarrow \infty.
\end{align*}
Integration of \eqref{e:A0} in time gives
\begin{align*}
A^0(x,\tau) - A^0(x,0) = \int_0^\tau A^0_\tau(x,s) ds = - \int_0^\tau w^0(x,s) ds.
\end{align*}
Therefore we conclude from the existence of the limits $A^\infty(x)$ and $B^\infty(x)$  for $\tau \rightarrow \infty$ that
\begin{align*}
  w^0(\cdot,\tau) \in L^1_+(0,\infty) \text{ for every } x \in \R.
\end{align*}
Then $\anull_\tau(x,\tau) = -w^0(x,\tau) + w^0(x+a,\tau)$ and integration over the interval $[0,\tau]$ gives
\begin{align*}
\anull(x,\tau) - f_I(x) &= - \int_0^\tau w^0(x,s) ds + \int_0^\tau w^0(x+a,s) ds.
\end{align*}
We know that $w^0(x,\cdot) \in L^1_+(0,\infty)$ and therefore we conclude for $\tau \rightarrow \infty$ that
\begin{align*}
\lim_{\tau \rightarrow \infty} \anull(x,\tau) = \ainfty(x), ~\lim_{\tau \rightarrow \infty} \bnull(x,\tau) = \binfty(x) \text{ and } \ainfty(x) \binfty(x) = w^\infty(x).
\end{align*}
Thus $\anull_\tau(x,\tau) \xrightarrow{\tau \rightarrow \infty} -w^\infty(x) + w^\infty(x+a) = 0$, and we conclude that $\ainfty(x) \binfty(x) \equiv 0$ on $\R$.
\end{proof}
\noindent The properties of $\ainfty$ and $\binfty$ can be can be summarised as follows:
\begin{enumerate}
\item[(P1)]\label{i:propone} $\ainfty(x) + \binfty(x+a) = f_I(x) + g_I(x+a)$.
\item[(P2)] $\displaystyle \underbrace{\sum_{l=0}^\infty \ainfty(x+al)}_{=: A^\infty(x)} - \underbrace{\sum_{l=0}^\infty \binfty(x-al)}_{=: B^\infty(x)} = \sum_{l=0}^\infty f_I(x+al) - \sum_{l=0}^\infty g_I(x-al)$.
\item[(P3)]\label{i:propthree} $\ainfty(x) \binfty(x) = 0$.
\item[(P4)]\label{i:propfour} $ \ainfty(x) \geq 0$, $\binfty(x) \geq 0$ on $\R$.
\end{enumerate}
Furthermore we set $C^\infty(x) = A^\infty(x) - B^\infty(x)$.
\begin{theorem}\label{t:fasttime}
Let $f_I, g_I$ satisfy \ref{a:A} and $h(x) := f_I(x) + g_I(x+a)$, $x\in\R$ satisfy
\begin{enumerate}
\item[(i)]  if for some $x_1 \in \R:~h(x_1) > 0$ and if there is an $l_1 \in \N$ with $h_I(x_1-a l_1) = 0$, then
$h_I(x_1 - a (l_1+1)) = 0$,
\item[(ii)] if for some $x_2 \in \R: ~ h(x_2) > 0$ and if there is an $l_2 \in \N$ with $h_I(x_2 + a l_2) = 0$, then 
$h_I(x_2 + a (l_2+1)) = 0$.
\end{enumerate}
Set $\Phi_I(x):= \sum_{l=0}^\infty f_I(x+al)-\sum_{l=0}^\infty g_I(x-al)$. Then the limiting functions $\ainfty(x) = \lim_{\tau \rightarrow \infty} \anull(x,\tau)$ and $\binfty(x) = \lim_{\tau \rightarrow \infty} \bnull(x,\tau)$ can be
identified as:
\begin{align}
\ainfty(x) = \Phi_I^+(x) - \Phi_I^+(x+a) \text{ and }\binfty(x) = \Phi_I^-(x) - \Phi_I^-(x-a).
\end{align}
\end{theorem}
\begin{proof}
The proof proceeds in analogy to the proof of Proposition \ref{p:identifylimit} and is split into two parts. First we show that the positivity of either $\ainfty$ or $\binfty$ at $x \in \R$ implies that the other variable is zero at $x$ and $x \pm a l$ as well. In the second step we
identify the limiting functions.
\begin{enumerate}
\item[(a)] Let $x_0 \in \R$ be such that $\ainfty(x_0) > 0$. Then $\binfty(x_0) = 0$ for all $l \in \N$, because of property (P3). Next we show by induction that $\binfty(x_0 - al) = 0,~\forall l \in \N$.
Therefore let for $l_1 \in \N$ hold $\binfty(x_0-al) = 0$ for all $l \in \lbrace 0,1,2,\ldots, l_1-1\rbrace$. Then we have
\begin{align}\label{e:con1}
\ainfty(x_0 - a l_1) + \underbrace{ \binfty(x_0-a(l_1-1))}_{=0} = h_I(x_0-a l_1)
\end{align}
and 
\begin{align*}
\ainfty(x_0) + \binfty(x_0+a) = h_I(x_0) > 0. 
\end{align*}
We distinguish between two cases:
\begin{enumerate}
\item [Case 1:] If $h_I(x_0 - al_1) > 0$, then $\ainfty(x_0 -al_1) > 0$ (because of \eqref{e:con1}). Hence we deduce, using property (P3), that $\binfty(x_0-a l_1) = 0$.
\item [Case 2:]  If $h_I(x_0-al_1) = 0$, then by assumption (i) of the theorem we have $h(x_0 - a(l_1 + 1)) = 0$. Property (P1) gives $\ainfty(x_0 - a(l_1+1)) + \binfty(x_0 - a l_1) = h_I(x_0-a(l_1+1)) = 0$
and we conclude that $\binfty(x_0 - al_1) = 0$. 
\end{enumerate}
\item[(b)] Let $x_0 \in \R$ with $\binfty(x_0) > 0$. From property (P3) we conclude that $\ainfty(x_0) = 0$. Again we use an induction argument to show that
$\ainfty(x_0+al) = 0$ for all $l\in \N$. Therefore for $l_2 \in \N$ let $\ainfty(x_0+al) = 0$ for all $l \in \lbrace 0,1,\ldots,l_2-1 \rbrace$. In this case
\begin{align}\label{e:con2}
\underbrace{\ainfty(x_0 + a(l_2-1))}_{=0} + \binfty(x_0 + al_2) = h_I(x_0 + a(l_2-1))
\end{align}
and
\begin{align*}
\ainfty(x_0-a) + \binfty(x_0) = h_I(x_0-a) > 0.
\end{align*}
Again we consider the two different cases.
\begin{enumerate}
\item [Case 1:] If $h_I(x+a(l_2-1)) > 0$, then $\binfty(x_0+al_2) > 0$ (because of \eqref{e:con2} ) and therefore $\ainfty(x_0 + al_2) = 0$ follows.
\item [Case 2:] If $h_I(x+a(l_2-1)) = 0$, then by assumption (ii) of the theorem we know that $h_I(x_0 + al_2) = 0$ and therefore $\ainfty(x_ß + al_2) + \binfty(x_0 + a(l_2+1)) = h_I(x+al_2) = 0.$
Thus $\ainfty(x_0 + a l_2) = 0$.
\end{enumerate}
\end{enumerate}
Now we identify the limiting functions $\ainfty = \ainfty(x)$ and $\binfty = \binfty(x)$. We choose an $x_0 \in \R$ such that $A^\infty(x_0) > 0$. Then there
exists $l_2 \in \N \cup \lbrace 0 \rbrace$, such that $\ainfty(x_0 + a l_2) > 0$. From (a) we deduce $\binfty((x_0 + a l_2) - al) = \binfty(x_0 - a (l-l_2)) = 0$ for all $l \in \N$.
Therefore $B^\infty(x_0) = 0$. Analogously $B^\infty(x_1) > 0$ implies $A^\infty(x_1) = 0$. Therefore $A^\infty(x) = C^\infty(x)^+$ and $B^\infty(x) = C^\infty(x)^-$ , i.e.
\begin{align*}
A^\infty(x) = \Phi_I^+(x) \text{ and } B^\infty(x) = \Phi_I^-(x).
\end{align*}
We conclude that
\begin{align*}
\ainfty(x) = \Phi_I^+(x) - \Phi_I^+(x+a) \text{ and } \binfty(x) = \Phi_I^-(x) - \Phi_I^-(x-a).\qedhere
\end{align*}
\end{proof}

Under the assumption of Theorem \ref{t:fasttime} on $f_I$ and $g_I$ we computed the unique solution candidate satisfying properties (P1)-(P4). Since existence of a solution of (P1)-(P4) follows already
from the $\tau \rightarrow \infty$ argument, we conclude that we have indeed calculated $\ainfty, \binfty$.

\noindent Note that Theorem \ref{t:fasttime} does not cover initial functions $f_I$ and $g_I \in C_0^\infty(\R)$, whose supports are separated by a distance larger than $a$ (see Figure \ref{f:initdata} (d)). If the distance between
the supports of $f_I$ and $g_I$ is smaller than $a$ or if they intersect, Theorem \ref{t:fasttime} is applicable. 
\begin{figure}
\begin{center}
\scalebox{0.4}{ \input{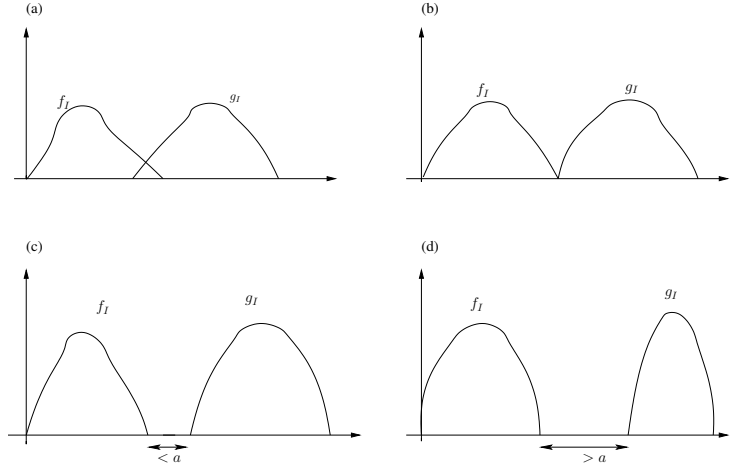}}
\end{center}
\caption{The initial datum depicted in figures (a)-(c) satisfies the assumptions of Theorem \ref{t:fasttime}. This is not the case for the initial datum depicted in Figure (d). }\label{f:initdata}
\end{figure}

\begin{rem}
Note that $\veps(x,\tau) = \aeps(x,\tau) + \beps(x+a,\tau)$ satisfies the heat equation \eqref{e:v}, thus
\begin{align*}
\lim_{\tau \rightarrow \infty} (\aeps, \beps) = 0 \text{ in } L^\infty(\R),
\end{align*} 
and, actually:
\begin{align*}
0 \leq \aeps(x,\tau) + \beps(x,\tau) \leq \frac{1}{\sqrt{4 \pi \varepsilon \tau}} \lVert h_I \rVert_{L^1(\R)} \quad \forall x \in \R, \forall \tau > 0.
\end{align*}
Therefore the limits $\tau \rightarrow \infty$ and $\varepsilon \rightarrow 0$ do not commute. Note that the initial layer problem \eqref{e:abeps} does in all generality not provide well-prepared initial data satisfying
assumption \ref{a:wellprep}. We conjecture that small diffusion helps in his process, i.e. that $\aeps,~\beps$ are close to well-prepared on large time scales $\tau = \tau(\varepsilon)$ with $\varepsilon \tau(\varepsilon)$
small (see Example 3 in Section \ref{s:numsim}).
\end{rem}

\section{Motivation by fast trading with small transaction rates}\label{s:kac}

\noindent In this section we consider the scaling limit $k\rightarrow \infty$, $a \rightarrow 0$ with $ka = c=const$. We rewrite system \eqref{e:boltzprice} as 
\begin{subequations}\label{e:rewrite1}
\begin{align}
f_t(x,t) &= c \frac{(fg)(x+a,t) - (fg)(x,t)}{a} + f_{xx}(x,t)\label{e:flimit} \\
g_t(x,t) &= c \frac{(fg)(x-a,t) - (fg)(x,t)}{a} + g_{xx}(x,t)\label{f:glimit}.
\end{align}
\end{subequations}
Formally it is clear that the limiting system as $k\rightarrow \infty$ and $a \rightarrow 0$ is
\begin{subequations}\label{e:limit}
\begin{align}
f^0_t(x,t) &= c (f^0g^0)_x(x,t) + f^0_{xx}(x,t) \\
g^0_t(x,t) &= -c (f^0g^0)_x(x,t) + g^0_{xx}(x,t).
\end{align}
\end{subequations}
Note that $u^0(x,t) = f^0(x,t) + g^0(x,t)$ satisfies again the heat equation $u^0_t(x,t) = u^0_{xx}(x,t)$ with initial datum $u^0(x,0) = f_I(x) + g_I(x)$.
The rigorous statement for this limit is stated in the following theorem:
\begin{theorem}
Let $k > 0$, $a > 0$, $ka = c$ and $f_I, g_I \in \mathcal{S}^+(\R)$. Then the weak limits $f^0$, $g^0$ satisfy system \eqref{e:limit} subject to $f^0(x,0)=f_I(x)$ and $g^0(x,0) = g_I(x)$.
\end{theorem}
\begin{proof}
We reiterate that $u(x,t) = f(x,t)+g(x+a,t)$ satisfies the heat equation \eqref{e:u}, therefore we can deduce that
\begin{align*}
\lVert u \rVert_{L^\infty(0,\infty;\R)} \leq K \text{ uniformly as } k\rightarrow \infty, a\rightarrow 0 \text{ with } ka = c.
\end{align*}
Multiplication of system \eqref{e:rewrite1} with $f$ and $g$ respectively gives
\begin{align*}
\frac{1}{2}\frac{d}{dt} \int_\R f^2 dx &= c \int_\R \frac{f(x+a,t)g(x+a,t)f(x,t) - f^2(x,t) g(x,t)}{a} dx - \int_\R f_x^2 dx\\
\frac{1}{2}\frac{d}{dt} \int_\R g^2 dx &= c \int_\R \frac{f(x+a,t)g(x+a,t)g(x,t) - f(x,t) g^2(x,t)}{a} dx - \int_\R g_x^2 dx.
\end{align*}
Now we look at the first term on the right hand side of the equation for $f$ and deduce
\begin{align*}
\int_\R& \frac{f(x+a,t)g(x+a,t)f(x,t) - f^2(x,t) g(x,t)}{a} dx  \\
& =\int_\R [f(x+a,t) f(x,t) \frac{g(x+a,t)-g(x,t)}{a} + f(x,t) g(x,t) \frac{f(x+a,t) - f(x,t)}{a}] dx\\
& \leq \lVert f \rVert_{L^\infty} \int_\R \lvert f(x,t)\rvert \lvert \frac{g(x+a,t)-g(x,a)}{a} \rvert dx + \lVert g \rVert_{L^\infty} \int_\R \lvert f(x,t)\rvert \lvert \frac{f(x+a,t) - f(x)}{a} \rvert dx\\
& \leq \frac{1}{2} \lVert f \rVert_{L^\infty} \frac{1}{\mu} \int_\R f^2(x,t) dx + \frac{1}{2} \lVert g \rVert_{L^\infty} \frac{1}{\mu}\int_\R f^2(x,t) dx \\
& \phantom{\leq} + \frac{\mu}{2} \lVert f \rVert_{L^\infty} \int_\R \lvert \frac{g(x+a,t)-g(x,t)}{a} \rvert^2 dx + \frac{\mu}{2} \lVert g \rVert_{L^\infty} \int_\R \lvert \frac{f(x+a,t)-f(x,t)}{a} \rvert^2 dx.
\end{align*}
The difference quotient for $f$ and $g$ can be estimated using the Fourier transform, i.e.
\begin{align*}
\int_\R \lvert \frac{f(x+a,t)-f(x,t)}{a}\rvert^2 dx = \int \lvert \hat{f}(\xi)\rvert^2 \lvert \frac{e^{2 \pi i a \xi} -1 }{a} \rvert^2 d\xi.
\end{align*}
Since $\lvert \frac{e^{2 \pi i a\xi}-1}{a}\rvert$ is bounded by $\lvert \xi \rvert^2$ we obtain that
\begin{align*}
\int_\R \lvert \frac{f(x+a,t)-f(x,t)}{a} \rvert^2 dx \leq L \int_\R \lvert f_x(x,t) \rvert^2 dx.
\end{align*}
Similar computations hold for the second equation in $g$. Now we choose $\mu = \frac{1}{2(\lVert f \rVert_{L^\infty} + \lVert g \rVert_{L^\infty}) L}$, then there exists a $\sigma > 0$ such that
\begin{align*}
\frac{1}{2} \frac{d}{dt} \int_\R (f(x,t)^2 + g(x,t)^2) dx + \sigma \int_\R (f_x(x,t)^2 + g_x(x,t)^2) dx \leq M.
\end{align*}
Thus
\begin{align*}
f,g \in \Ltwoloc(0,\infty; H^1(\R)) \text{ uniformly for } a \rightarrow 0.
\end{align*}
By the Aubin Lions lemma, \cite{Showalter}, we conclude that $f \rightarrow f^0$ and $g \rightarrow g^0$ as $a \rightarrow 0$. Thus we can pass to the limit in system \eqref{e:rewrite1}, making the formal limit rigorous.
\end{proof}

\begin{rem}
  Consider the price formation FB problem \eqref{e:lasrylionsdiff} subject to the initial condition $v(x,0) = f_I(x)-g_i(x)$, where $f_I$ and $g_I$ satisfy \ref{a:A} and \ref{a:wellprep}. Being interested in the
  limit as the transaction fee $a$ tends to zero we denote $v = v^a, ~f = f^a$ and $g = g^a$. The initial data $f_I$ and $g_I$ are assumed to be independent of $a$. It is an easy exercise to show that the weak
limits $v^a \xrightarrow{a\rightarrow 0} v^0$, $f^a \xrightarrow{a\rightarrow 0} f^0$ and $g^a \xrightarrow{a\rightarrow 0} g^0$ exist and that
\begin{align*}
  v^0 = -\frac{d}{dx}\lvert F^0\rvert, f^0 = (v^0)^+ = -\frac{d}{dx} ( F^0 )^+, g^0 = (v^0)^- = \frac{d}{dx} (F^0)^-
\end{align*}
holds, where $F^0$ satisfies the IVP for the heat equation
\begin{align*}
&F^0_t = F^0_{xx}, \quad  x \in \mathbb{R}, t > 0\\
&F^0(x,0) = \int_0^\infty f_I(x+y)dy -\int_0^\infty g_I(x-y) dy.
\end{align*}
Using the methods of \cite{CMP, CMW} we conclude the existence of a unique continuous function $p^0 = p^0(t)$ (the price of the limiting system) such that
\begin{align*}
F^0(p^0(t),t) = 0.
\end{align*}
The parabolic Hopf lemma implies $\frac{\partial}{\partial x} F^0(p^0(t),t) < 0$. Therefore the function $v^0(\cdot,t)$ has a jump discontinuity at $x = p^0(t)$, with
\begin{align*}
\lim_{x \rightarrow p^0(t)^+} v^0(x,t)  = -\lim_{x \rightarrow p^0(t)^-} v^0(x,t) < 0.
\end{align*}
We conclude that $v^a(x,t)$ has an internal layer of rapid transition at $p = p^a(t)$  for $a$ small. Moreover it follows that the limits $k\rightarrow \infty$ (from the Boltzmann type model to the Lasry Lions FBP)
and $a \rightarrow 0$ (from the Lasry Lions FBP to the  limit as discussed in the remark) consecutively do not give the same result as $k \rightarrow \infty$, $a\rightarrow 0$ and $ka=c > 0$ in the
Boltzmann type model.\\
We conclude that for markets with many transactions with small fees, the Boltzmann price formation model is very sensitive to the relative sizes of 
the transaction frequencies and transaction fees. We refer to the next section for numerical illustrations.
\end{rem}

\section{Numerical simulations}\label{s:numsim}

\noindent In this section we illustrate the behaviour of the Boltzmann type price formation model \eqref{e:boltzprice} on a bounded domain $\Omega$. We supplement system \eqref{e:boltzprice} with  homogeneous Neumann boundary
conditions, hence the total mass of buyers and vendors is conserved over time. We simulate the behaviour of solutions of \eqref{e:boltzprice} for different values of $k$ and compare them to the solutions of \eqref{e:lasrylions}.
Furthermore we perform numerical simulations in the case of not well prepared initial data and study the evolution of the initial layer on the fast time scale, see also Section \ref{s:initlayer}.\\

\noindent All simulations are performed on the interval $\Omega = (0,1)$, where $1$ corresponds to the scaled maximum price of the traded good. The domain $\Omega$ is split into $N$ intervals of size $h$. 
We denote the discrete grid points by  $x_i$ and the time steps by $\Delta t$. Then the discrete solutions $f^n_j$ and $g^n_j$ correspond to the functions $f$ and $g$ at time $t^n = n \Delta t$ and $x_j = j h$.\\
\noindent The simulations of the free boundary problem \eqref{e:lasrylions} are based its formulation as an IVP for the heat equation, i.e.    
\begin{subequations}\label{e:V}
\begin{align}
&V_t(x,t) = V_{xx}(x,t) \text{ for all } x \in \Omega\\
&V_x(0,t) = V_x(a,t) \text{ and } V_x(1,t) = V_x(1-a,t),\label{e:Vbc}
\end{align}
\end{subequations}
with $v = v(x,t)$ being a solution to \eqref{e:lasrylionsdiff} and $V$ given by
\begin{align*}
V(x,t) &=
\begin{cases}
\phantom{-}\sum_{n=0}^\infty v^+(x + al,t)\\
-\sum_{n=0}^\infty v^-(x+al,t),
\end{cases}
\end{align*}
as in \cite{CMP}. Note that \eqref{e:Vbc} corresponds to the transformed homogeneous Neumann boundary conditions.  System \eqref{e:V} is solved using an implicit in time finite difference method. The solutions $f$ and $g$ to \eqref{e:lasrylions} can be calculated from
\begin{align*}
v(x,t) = V(x,t) - V^+(x+a,t) + V^-(x-a,t).
\end{align*}
The price $p = p(t)$ corresponds to the zero level set of $V = V(x,t)$.\\
System \eqref{e:boltzprice} is solved using a semi-implicit in time discretization, i.e. the diffusion is discretised implicitly in time, the ``collision'' terms explicitly. This results in the 
following discretization:
\begin{subequations}\label{e:discrete}
\begin{align}
\frac{f^n_j - f^{n-1}_j}{\Delta t} &= -k f^{n-1}_j g^{n-1}_j + k f^{n-1}_{j+a} g^{n-1}_{j+a} + \frac{1}{h^2}(f^n_{j+1}-2f^n_j + f^n_{j-1})\\
\frac{g^n_j - g^{n-1}_j}{\Delta t} &= -k f^{n-1}_j g^{n-1}_j  +k f^{n-1}_{j-a} g^{n-1}_{j-a} + \frac{1}{h^2}(g^n_{j+1}-2g^n_j + g^n_{j-1}).
\end{align}
\end{subequations}
Thus the solution of \eqref{e:discrete} corresponds to solving the linear system 
\begin{align*}
(I + \Delta t A) \binom{f^n}{g^n} = (I + \Delta t B) \binom{f^{n-1}}{g^{n-1}},   
\end{align*}
where the matrices $I$, $A,$ and $B$ are to the identity, the discrete Laplacian and the discrete ``collision'' term respectively.

\subsection*{Example 1: Comparison of the models for large trading rates $k$} 
\noindent We compare the behaviour of solutions of \eqref{e:boltzprice} with \eqref{e:lasrylions} using the same initial data $f_I = f_I(x)$ and $g_I = g_I(x)$. Therefore we choose an initial datum, which satisfies the
compatibility condition $\lambda(t) = -\frac{\partial f}{\partial x}(p(0),0) = \frac{\partial g}{\partial x}(p(0),0)$, namely
\begin{align*}
&f_I(x) = \begin{cases}
1 & \text{ if } x \leq 0.5\\
-10 x + 6 & \text{ if } 0.5 < x < 0.6\\
0 & \text{ otherwise} 
\end{cases}
&g_I(x) = \begin{cases}
10 x - 6 &\text{ if } x > 0.6\\
0 & \text{ otherwise}.\\
\end{cases}
\end{align*}
Since \eqref{e:boltzprice} converges to \eqref{e:lasrylions} as $k\rightarrow\infty$, we set $k = 10^6$. The other parameters are set to: $a = 10h$, $\sigma = 1$, $\Delta t = \frac{1}{k} = 10^{-6}$ and $h = 0.002$. 
Figure \ref{f:ex1} shows the evolution of the price and the distribution of buyers and vendors at time $t=1$ for both models. We observe that the buyer-vendor distribution as well as the evolution of the price agree very well. 
% \cite{CMP} showed that the price evolves like $p(t) \approx \sqrt{t} q$, where the constant $q$ depends on the initial mass of buyers and vendors. This behaviour can be observed in both models.
\begin{figure}
\begin{center}
  \subfigure[Buyer and vendor distribution at time $t=1$]{\includegraphics[width=0.4\textwidth]{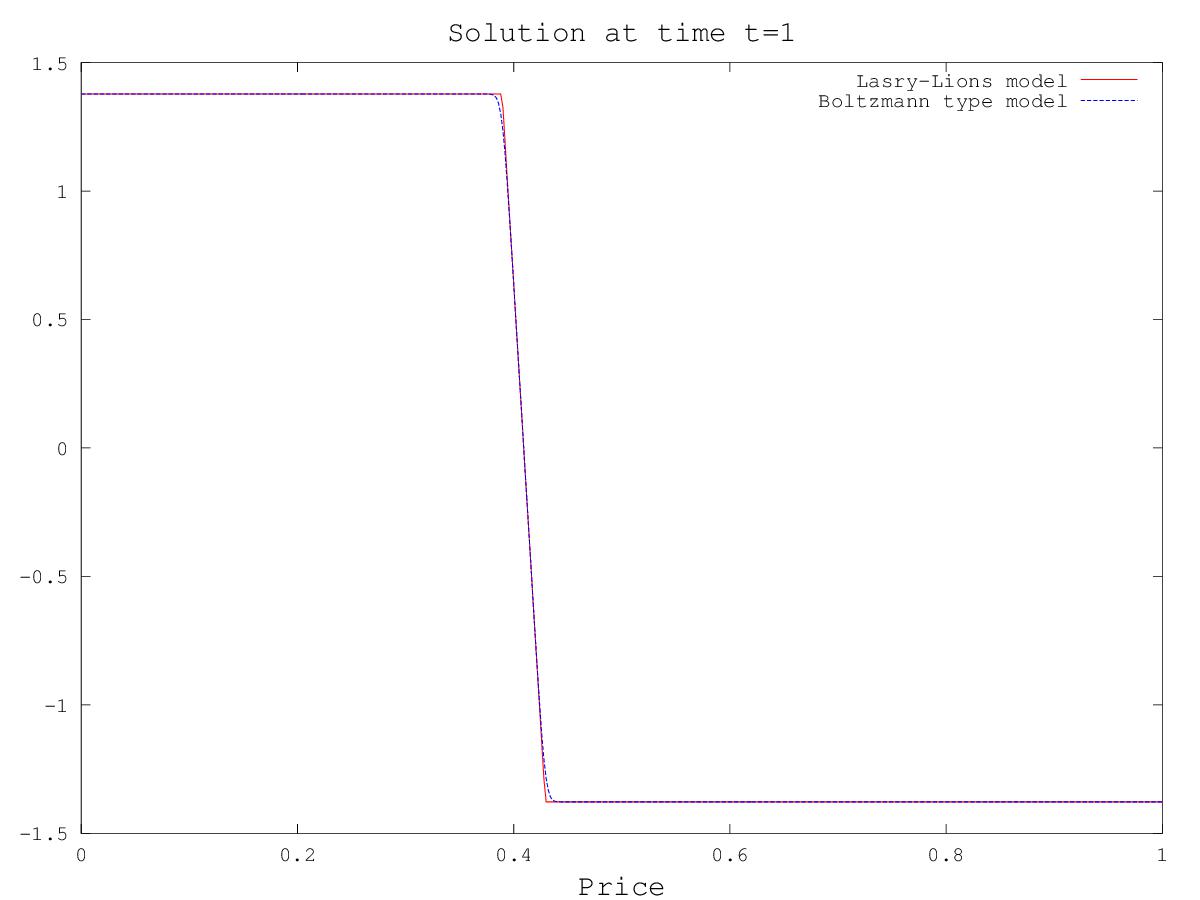}}\hspace*{0.5cm}
  \subfigure[Evolution of the price $p(t)$ in time]{\includegraphics[width=0.4\textwidth]{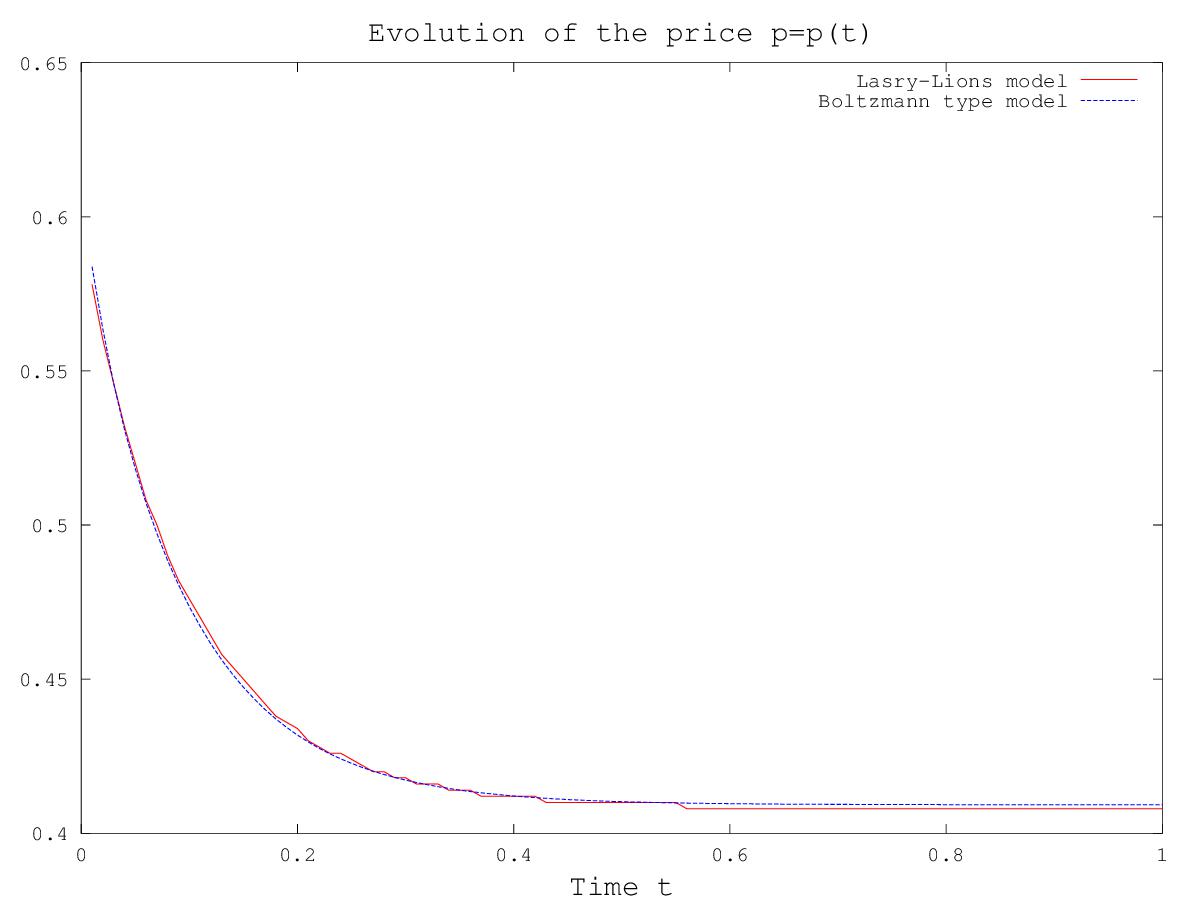}}
\end{center}
\caption{Comparison of the price formation models for conforming initial data}\label{f:ex1}
\end{figure}

\subsection*{Example 2: Behaviour of solutions of \eqref{e:boltzprice} for different trading rates}
\noindent Next we illustrate how solutions of \eqref{e:boltzprice} change for different values of $k$. The initial datum is set to
\begin{align*}
&f_I(x) = \begin{cases}
15 (x-0.3)(0.5-x) & \text{ if } 0.3 \leq x \leq 0.5\\
0 & \text{ otherwise }
\end{cases}\\
&g_I(x) = \begin{cases}
15(0.55-x)(x-0.8) &\text{ if } 0.55 \leq x \leq 0.8\\
0 & \text{ otherwise}.\\
\end{cases}
\end{align*}
The simulations were performed with the following parameters: $h=10^{-3}, \tau = \frac{1}{k}$, $k = 10^5 \text{ and } 10^6$, $a = 10h$.
The final distribution of buyers and vendors as well as the evolution of the price is depicted in Figure \ref{f:ex2}. We observe that the value $c = f(x,0.5) = f(x,0.5)$
decreases if $k$ increases, also the evolution of the price towards its stationary value is slower. The support of the trading distribution \eqref{e:trade} decreases for 
larger $k$ and converges to a Dirac $\delta$ as in \eqref{e:lasrylions}.
\begin{figure}
\begin{center}
  \subfigure[Buyer and vendor distribution at time $t=0.5$ for $k=10^5, 10^6$.]{\includegraphics[width=0.45\textwidth]{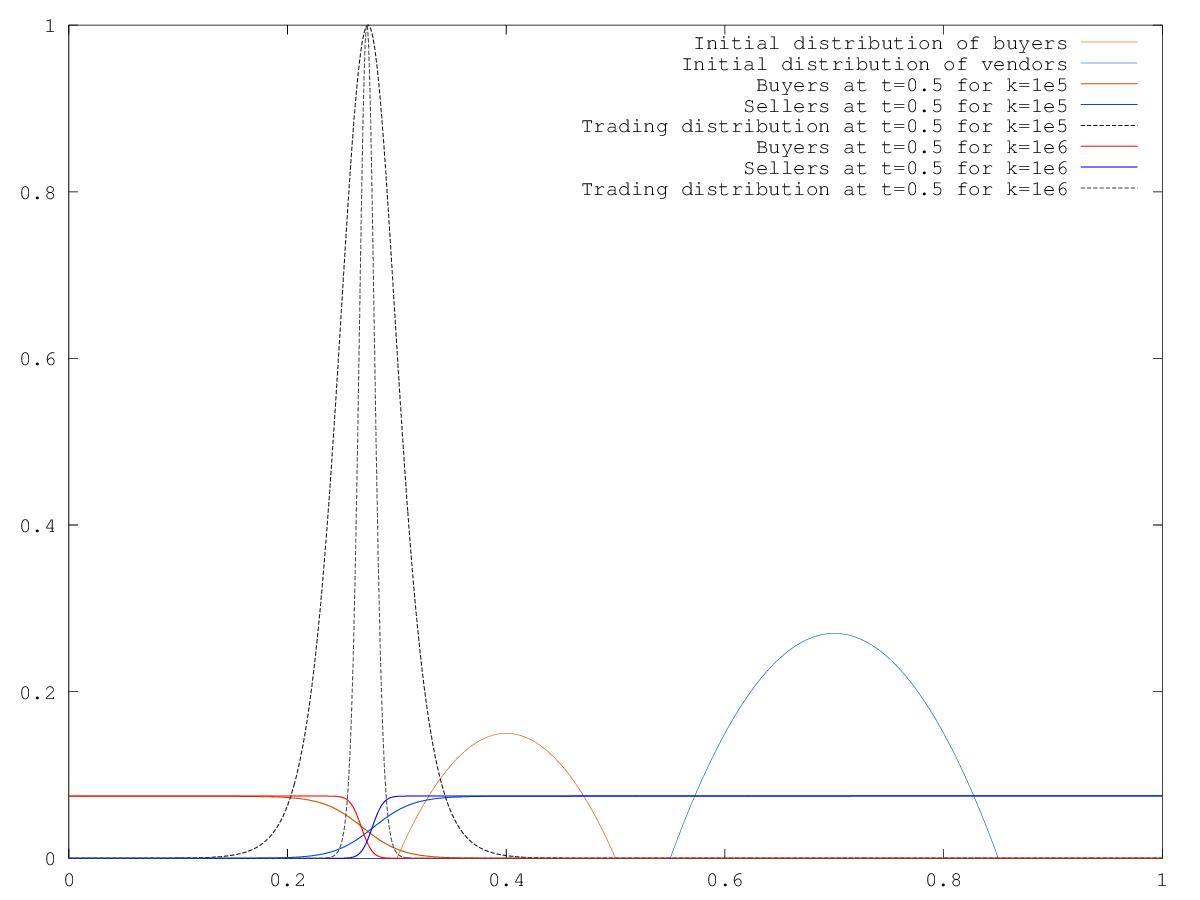}}\hspace*{0.5cm}
  \subfigure[Evolution of the price $p(t)$ in time for $k=10^5,10^6$]{\includegraphics[width=0.45\textwidth]{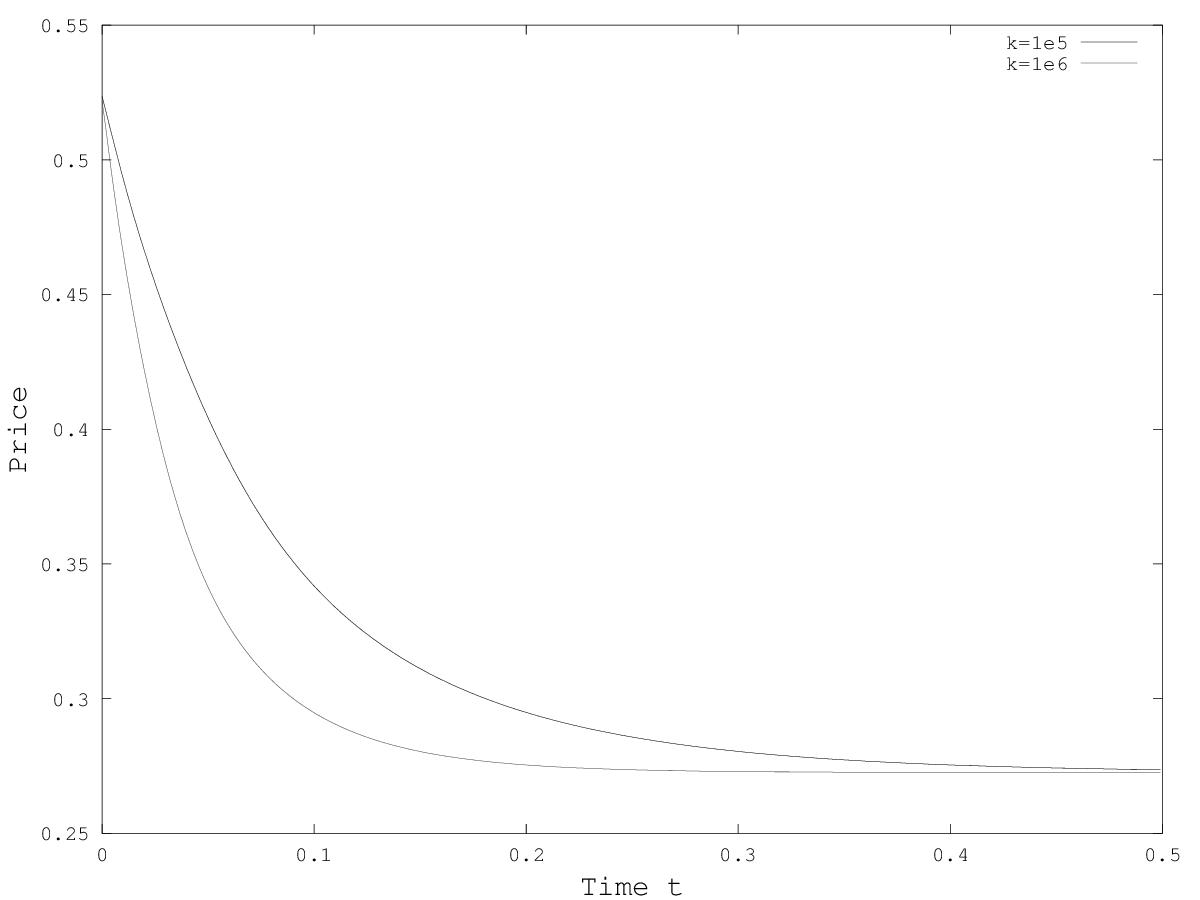}}
\end{center}
\caption{Comparison of the Boltzmann type model for different values of $k$}\label{f:ex2}
\end{figure} 

\subsection*{Example 3: Not well prepared initial datum on the fast time scale} 

\noindent In this example we study the behaviour of solutions in the case of a not well prepared (in the sense of assumption \ref{a:wellprep}) initial data $f_I$ and $g_I$  on the fast time scale. Hence we simulate system \eqref{e:abeps} 
with an initial datum, where the density of vendors and buyers are switched, i.e. the initial data are not well prepared in the sense of 
assumption \ref{a:wellprep}. We set
\begin{align*}
&f_I(x) = \begin{cases}
15(x-0.65)(0.95-x) & \text{ if } 0.65 \leq x \leq 0.95\\
0 & \text{ otherwise} 
\end{cases}\\
&g_I(x) = \begin{cases}
12(x-0.25)(0.5-x)  &\text{ if } 0.25 \leq x \leq 0.5\\
0 & \text{ otherwise}.
\end{cases}
\end{align*}
We choose a spatial resolution of $h = 2 \times 10^{-3} $, the time steps are $\Delta t = 2 \times 10^{-4}$, the transaction rate $k=5 \times 10^2$ and the transaction cost $a = 10 h$.  
The evolution of $\sup f(x) > h$ and the $\inf g(x) > h$ as well as the final distribution of buyers and vendors is depicted in Figure \ref{f:ex3}.
We observe that even though the initial data is not well prepared $\aeps$ and $\beps$ converge to a well prepared solutions in the sense of
assumption \ref{a:wellprep}.
\begin{figure}
\begin{center}
 \subfigure[Buyer and vendor distribution at time $t=0$ (dashed line) and $t=1$ (solid line)]{\includegraphics[width=0.45\textwidth]{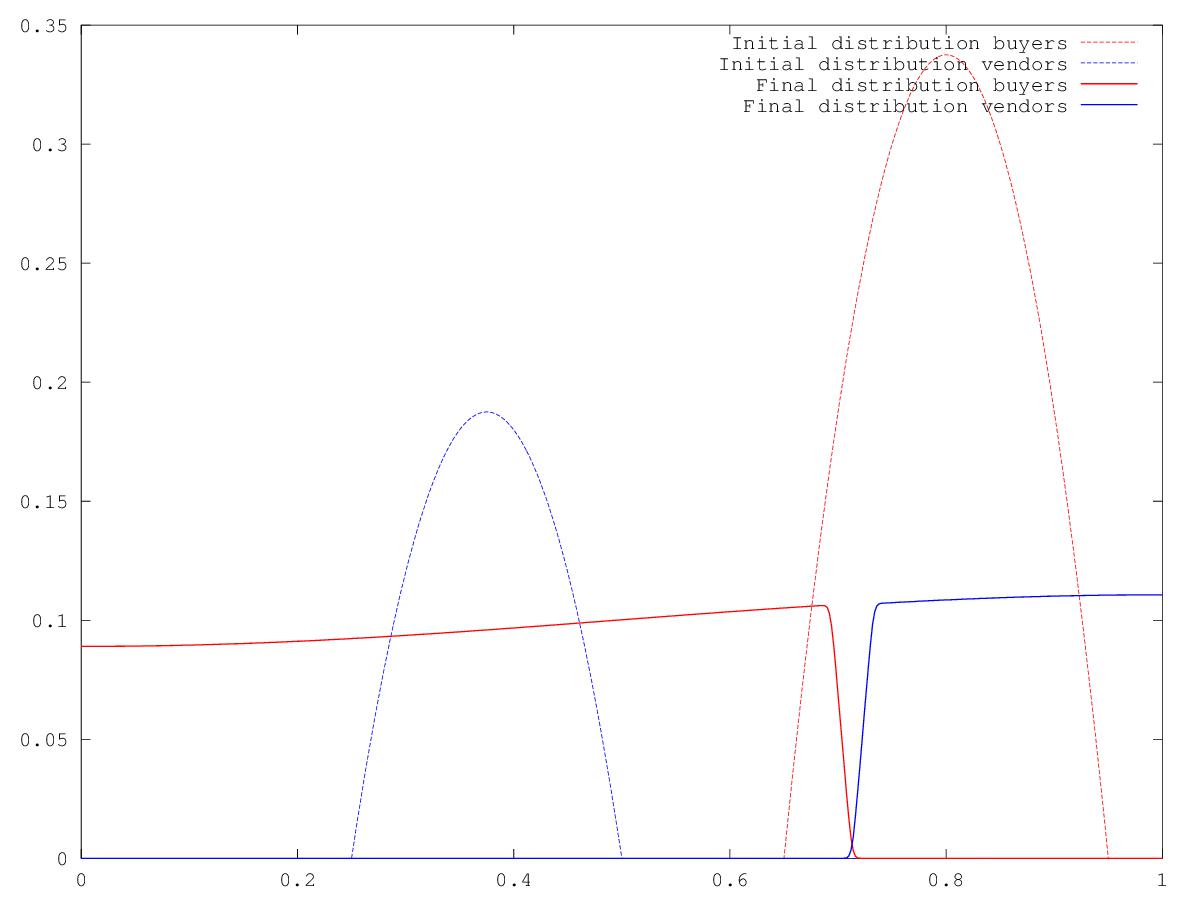}}\hspace*{0.5cm}
 \subfigure[Evolution of $\sup f(x)$ and $\inf g(x)$ in time]{\includegraphics[width=0.45\textwidth]{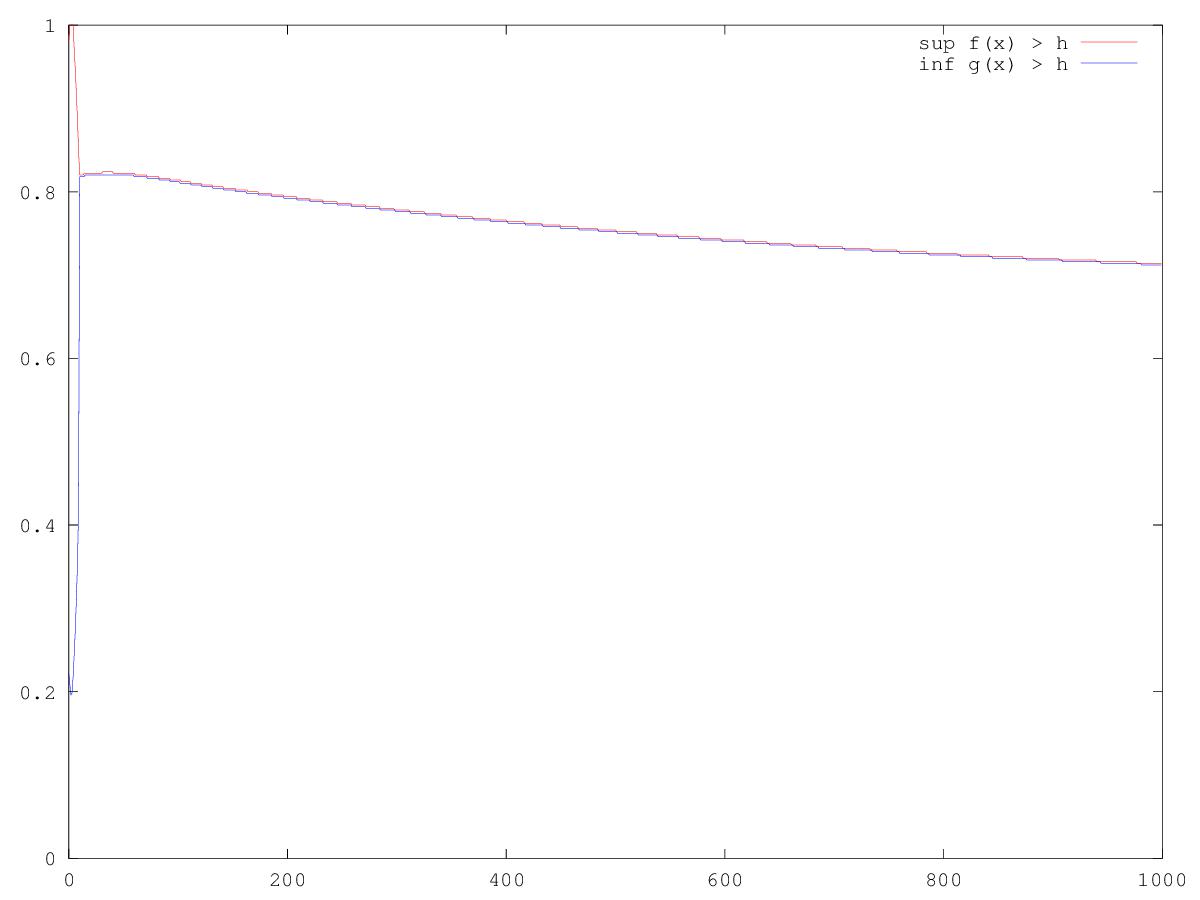}}
\end{center}
\caption{Non-conforming initial data on the fast time scale}\label{f:ex3}
\end{figure}  

\subsection*{Example 4: Scaling limit}
Finally we consider the scaling limit discussed in Section \ref{s:kac}. We choose a large domain $\Omega = [0,20]$ and an initial datum of the form
\begin{align*}
&f_I(x) = \begin{cases}
1 & \text{ if } 9 \leq x \leq 9.5\\
-2 x + 20 & \text{ if } 9.5 < x < 10\\
0 & \text{ otherwise} 
\end{cases}
&g_I(x) = \begin{cases}
2 x - 20 &\text{ if } 10 \leq x > 11\\
0 & \text{ otherwise},\\
\end{cases}
\end{align*}
to avoid any interference with boundary conditions. We observe the evolution of the price for the following set of parameters for \eqref{e:boltzprice}
\begin{align*}
a = h = 2 \times 10^{-5}, k = 5 \times 10^4, \sigma = 1 \text{ and } \tau = 2 \times 10^{-5}.
\end{align*}
This parameter set corresponds to the case where $ak = \mathcal{O}(1)$, i.e. the scaling limit in Section \ref{s:kac}. Figure \ref{f:ex4} shows the solutions
of \eqref{e:boltzprice} and \eqref{e:lasrylions} at time $t=1$ and the evolution of the price $p = p(t)$. We observe that the consecutive limits $k \rightarrow \infty$ and 
$a \rightarrow 0$ in Lasry\&Lions do not give the same result as the scaling limit discussed in Section \ref{s:kac}. Also the layer of fast transition in the Lasry-Lions solutions $f$ and $g$ 
is clearly visible in Figure 4a, as mentioned in Remark 2 of Section \ref{s:kac}.
\begin{figure}
\begin{center}
 \subfigure[Buyer and vendor distribution of \eqref{e:boltzprice} and \eqref{e:lasrylions}  at time $t=1$]{\includegraphics[width=0.45\textwidth]{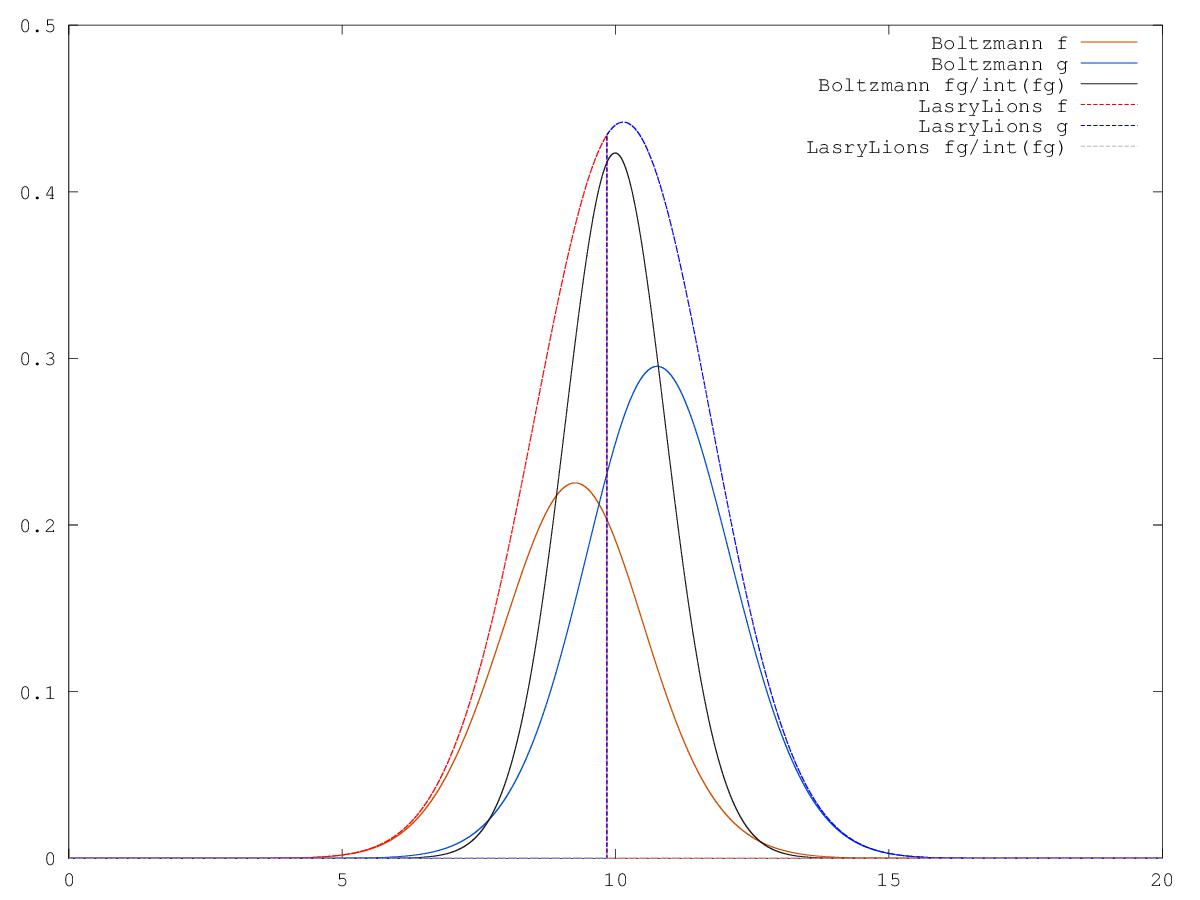}}\hspace*{0.5cm}
 \subfigure[Evolution of $p=p(t)$ in time]{\includegraphics[width=0.45\textwidth]{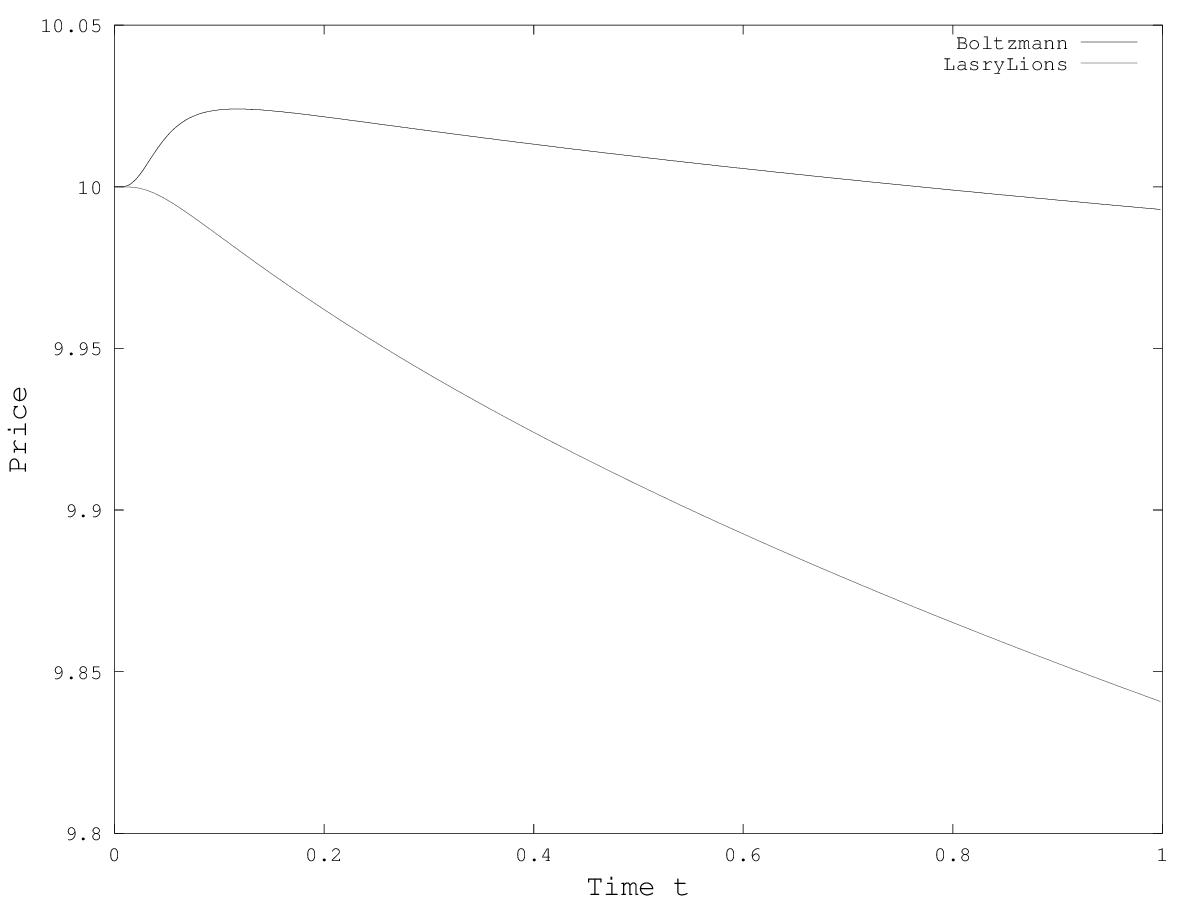}}
\end{center}
\caption{Non-conforming initial data on the fast time scale}\label{f:ex4}
\end{figure}

\section*{Acknowledgements}
LC was partially supported by a grant of  the DMS division of the NSF. PAM expresses his gratitude to the Humboldt Foundation for awarding the Humboldt Research Award to him, which allowed him to spend
time with Martin Burger's research group in M\"unster, where this research was initiated. PAM also acknowledges support from the Paris Foundation of Mathematics. MTW acknowledges support from the Austrian 
Science Foundation FWF via the Hertha-Firnberg project T456-N23. We thank Bertram D\"uring (University of Sussex) for the useful hints to literature.

\bibliographystyle{royalagain}
\bibliography{priceboltzmann}

\end{document}